\documentclass[a4paper,12pt]{article}

\makeatletter
\@addtoreset{footnote}{page}
\makeatother

\usepackage{enumitem}

\usepackage{amsthm}
\usepackage{amsmath,amssymb,latexsym,amsfonts,mathrsfs}

\renewcommand{\Re}{\mathop{\rm Re}}

\newcommand{\eps}{\ensuremath{\varepsilon}}

\renewcommand{\tilde}{\widetilde}

\newcommand{\up}{\ensuremath{\uparrow}}

\newcommand{\bC}{\ensuremath{\mathbb{C}}}

\newcommand{\bE}{\ensuremath{\mathbb{E}}}

\newcommand{\bP}{\ensuremath{\mathbb{P}}}
\newcommand{\bQ}{\ensuremath{\mathbb{Q}}}
\newcommand{\bR}{\ensuremath{\mathbb{R}}}

\newcommand{\cA}{\ensuremath{\mathcal{A}}}

\newcommand{\cC}{\ensuremath{\mathcal{C}}}
\newcommand{\cD}{\ensuremath{\mathcal{D}}}

\newcommand{\cF}{\ensuremath{\mathcal{F}}}

\newcommand{\cL}{\ensuremath{\mathcal{L}}}

\newcommand{\cT}{\ensuremath{\mathcal{T}}}

\theoremstyle{plain}
\newtheorem{Thm}{Theorem}[section]

\newtheorem{Lem}[Thm]{Lemma}
\newtheorem{Prop}[Thm]{Proposition}
\newtheorem{Cor}[Thm]{Corollary}

\theoremstyle{definition}

\newtheorem{Def}[Thm]{Definition}
\newtheorem{Rem}[Thm]{Remark}

\numberwithin{equation}{section}

\makeatletter
\renewcommand\section{\@startsection {section}{1}{\z@}%
                                   {-3.5ex \@plus -1ex \@minus -.2ex}%
                                   {2.3ex \@plus.2ex}%
                                   {\normalfont\large\bf}}
\makeatother

\makeatletter
\renewcommand\subsection{\@startsection {subsection}{1}{\z@}%
                                   {-3.5ex \@plus -1ex \@minus -.2ex}%
                                   {2.3ex \@plus.2ex}%
                                   {\normalfont\normalsize\bf}}
\makeatother

\usepackage[sort]{natbib}
\setcitestyle{authoryear,comma,open={(},close={)}} 

\usepackage{hyperref}
\usepackage{mathtools}
\mathtoolsset{showonlyrefs}

\begin{document}
\begin{center}
	\textbf{\Large
		Conditioning to avoid zero via a class of concave functions for one-dimensional diffusions
	}
\end{center}
\begin{center}
	Kosuke Yamato (University of Osaka)
\end{center}
\begin{center}
	{\small \today}
\end{center}

\begin{abstract}
	For one-dimensional diffusions on the half-line, we study a specific type of conditioning to avoid zero.
	We introduce supermartingales defined via concave functions with respect to the scale function.
	A conditioning is formulated through the exit times of the supermartingale, and its existence is shown.
	We also investigate the absolute continuity relations of the limit laws at time infinity.
\end{abstract}


\section{Introduction}

Let $X = ( (X_{t})_{t \geq 0}, (\cF_{t})_{t \geq 0}, (\bP_{x})_{x \geq 0} )$ be a non-singular one-dimensional diffusion on the half-line killed on hitting zero defined on a measurable space $(\Omega, \cF)$.
In the present paper, we study a specific type of conditioning of $X$ to avoid zero.

Given a non-decreasing sequence of random times $\tau = (\tau_{\lambda})_{\lambda \geq 0}$ diverging to infinity, which we call a \textit{clock}, we can consider the conditional limit distribution of $X$ to avoid zero, if it exists:
\begin{align}
	\bP_{x}[A \mid T_{0} > \tau_{\lambda}] \xrightarrow[\lambda \to \infty]{} \nu_{s,x}(A) \quad \text{for every $s \geq 0$, $x > 0$ and $A \in \cF_{s}$}, \label{conditionalLimit}
\end{align}
where $T_{0}$ is the first hitting time of zero and $\nu_{s,x}$ is the limit distribution.

We are interested in the case where the family of limit distributions forms a Markov process.
If there exist distributions $(\bQ_{x})_{x > 0}$ such that $((X_{t})_{t \geq 0},(\cF_{t})_{t \geq 0},(\bQ_{x})_{x > 0})$ is a Markov process on $(\Omega,\cF_{\infty})$, where $\cF_{\infty} := \sigma\left( \bigcup_{t \geq 0}\cF_{t} \right)$, and the following holds:
\[
\nu_{t,x}(A) = \bQ_{x}[A] \quad (t \geq 0, \ A \in \cF_{t}),
\]
we say that $((X_{t})_{t \geq 0},(\cF_{t})_{t \geq 0},(\bQ_{x})_{x > 0})$ is the \textit{process $X$ conditioned to avoid zero by the clock $\tau$}.

The most famous example of such a conditioning is for Brownian motion.
Let $B = ((B_{t})_{t \geq0},(\cF_{t})_{t \geq 0},(\bP^{B}_{x})_{x \in \bR})$ be the canonical representation of Brownian motion with its natural filtration $(\cF_{t})_{t \geq 0}$.
It is classically known that
\begin{align}
	\lim_{t \to \infty} \bP^{B}_{x}[A \mid T_{0}^{B} > t] = \bP^{\mathrm{Bes}}_{x}[A] \quad (x > 0,\ s \geq 0, \ A \in \cF_{s}), \label{}
\end{align}
where $T_{0}^{B} := \inf\{ t \geq 0 \mid B_{t} = 0 \}$ and $\bP_{x}^{\mathrm{Bes}}$ denotes the underlying probability distribution of three-dimensional Bessel process starting from $x$ (see e.g., \cite{McKeanAvoidZero} for the details).
In this case, the clock is the constant clock $\tau_{\lambda} := \lambda$.

This classical result has been extended to a wider range of processes, possibly in the more general context of $Q$-processes (see e.g., \cite{McKeanAvoidZero,AsymptoticLawsNonAbsorption,YanoYanohTransform,LambertQSD,ChampagnatVillemonais-ExponentialConvergence,ChampagnatVillemonais-Q-process}).
The cases other than the constant clock have also been studied extensively, including the more general problem of the local time penalization (see e.g., \cite{PenalisingBrownianPaths,ChaumontDoneyStayPositive,PantiAvoidZero,ProfetaYanoYano,TakedaYanoPenalizationLevy}).
An interesting feature when considering the various clocks is that the limit distribution may depend on the choice of clocks.

In the present paper, we introduce a new clock and prove the existence of the conditioning to avoid zero for the clock.
Let $X$ be a $\frac{d}{dm}\frac{d^{+}}{ds}$-diffusion, a diffusion whose  scale function is $s$ and speed measure is $dm$.
A function $\rho$ is said to be concave with respect to $s$ when it is right-differentiable with respect to $ds$ and its right-derivative $\frac{d^{+}}{ds}\rho$ is non-increasing.
We introduce a class of supermartingales induced by the concave functions.
We then define a clock by the exit times of the supermartingale.
Theorem \ref{thm:conditionalLimit} establishes the existence of the conditioning by the clock, which is one of the main results of the present paper.

The other main objective is to investigate the absolute continuity relation of the conditional limits on $\cF_{\infty}$.
Let $((X_{t})_{t \geq 0},(\cF_{t})_{t \geq 0},(\bQ^{(1)}_{x})_{x > 0})$ and $((X_{t})_{t \geq 0},(\cF_{t})_{t \geq 0},(\bQ^{(2)}_{x})_{x > 0})$ be two conditionings obtained by the result mentioned above.
As Theorem \ref{thm:conditionalLimit} shows, the conditioning is given by a change of measure of $((X_{t})_{t \geq 0},(\cF_{t})_{t \geq 0},(\bP_{x})_{x > 0})$ by a supermartingale.
If the corresponding supermartingales are actually martingales, the distributions $\bQ^{(1)}_{x}$ and $\bQ^{(2)}_{x}$ are mutually absolutely continuous on $\cF_{t}$ for any $t>0$ and $x>0$.
Specifically, we have
\begin{align}
	\bQ^{(i)}_{x}[A] = \bQ^{(j)}_{x}\left[ \frac{M^{(i)}_{t}}{M^{(j)}_{t}},A \right] \quad (A \in \cF_{t},\ i, j \in \{1,2\}) \nonumber
\end{align}
for the corresponding martingale $(M^{(i)}_{t})_{t \geq 0}$.
The point is that this relation may not be extended to $t = \infty$.
In fact, it is possible that $\bQ^{(1)}_{x}$ and $\bQ^{(1)}_{x}$ are singular on $\cF_{\infty}$.
In Section \ref{section:absoluteContinuity}, we study the conditions under which the conditionings are mutually absolutely continuous or singular on $\cF_{\infty}$.

This study is motivated by \cite{Ratesofdecay}.
We briefly recall their results.
For the details on general facts on one-dimensional diffusions used below, see Section \ref{section:preliminary} and \ref{section:generalResults}.
Let $((X_{t})_{t \geq 0},(\cF_{t})_{t \geq 0}, (\bP_{x})_{x \in [0,\infty)})$ be the canonical representation of a natural-scale $\frac{d}{dm}\frac{d^{+}}{dx}$-diffusion on $[0,\infty)$ with the right-continuous induced filtration $(\cF_{t})_{t \geq 0}$.
We suppose that the boundary $0$ is regular and the boundary $\infty$ is natural in the sense of Feller (see Section \ref{section:classificationOfTheBoundary}).
This implies that $X$ certainly hits zero:
\begin{align}
	\bP_{x}[T_{0} < \infty] = 1 \quad (x > 0). \label{eq08}
\end{align}
Define $u = \psi_{q} \ (q \in \bC)$ as the unique solution of the following integral equation:
\begin{align}
	u(x) = x + q \int_{0}^{x}dy \int_{0}^{y}u(z)dm(z). \label{eq12}
\end{align}
Let $\lambda_{0}$ be the bottom of the $L^{2}(m)$-spectrum of the operator $-\frac{d}{dm}\frac{d^{+}}{dx}$.
It is known that when $\lambda_{0} > 0$, the function $\psi_{-\lambda}(x) \ (\lambda \in [0,\lambda_{0}])$ is positive on $(0,\infty)$ and $\lambda$-invariant, i.e., 
\[
\psi_{-\lambda}(x) = \mathrm{e}^{\lambda t}\bE_{x}[\psi_{-\lambda}(X_{t}),T_{0} > t]
\]
(see \citet[Lemma 4]{Ratesofdecay} and Proposition \ref{prop:monotonicityOfPsi}).
This fact allows us to consider the $h$-transform by $\psi_{-\lambda}$, which is given by a change of measure of $X$ by the martingale
\begin{align}
	M^{[\psi_{-\lambda}]}_{t} := \mathrm{e}^{\lambda t} \frac{\psi_{-\lambda}(X_{t})}{\psi_{-\lambda}(X_{0})}1\{ T_{0} > t \}. \label{eq13}
\end{align}
That is, the $h$-transform $((X_{t})_{t \geq0}, (\cF_{t})_{t \geq 0}, (\bQ_{x})_{x > 0})$ is given by
\begin{align}
	\bQ_{x}[A] = \bE_{x}[M_{t}^{[\psi_{-\lambda}]},A] \quad (A \in \cF_{t}).
\end{align}
In \cite{Ratesofdecay}, they introduced a clock by an exit time of the martingale $M^{[\psi_{-\lambda}]}$ and showed that the conditioning to avoid zero by the clock exists, and the limit process is given by the $h$-transform:

\begin{Thm}[{\citet[Theorem 3]{Ratesofdecay}}] \label{thm:conditionalLimitPrevious}
	Suppose $\lambda_{0} > 0$ and let $\lambda \in (0,\lambda_{0})$.
	Define the clock $(S_{r})_{r > 0}$ by
	\begin{align}
		S_{r} := \inf \{ t > 0 \mid M^{[\psi_{-\lambda}]}_{t} \geq r/\psi_{-\lambda}(X_{0}) \}. \label{eq06}
	\end{align}
	Then we have for every $x > 0$, $s > 0$ and $A \in \cF_{s}$
	\begin{align}
		\lim_{r \to \infty} \bP_{x}[A \mid T_{0} > S_{r} ] = \bE_{x}\left[ M^{[\psi_{-\lambda}]}_{s}, A \right]. \label{eq05}
	\end{align}
\end{Thm}

One of our main results (Theorem \ref{thm:conditionalLimit}) generalizes this theorem.
We note that the function $\psi_{-\lambda}$ is concave since it satisfies 
\begin{align}
	\frac{d}{dm}\frac{d^{+}}{dx}\psi_{-\lambda} = -\lambda\psi_{-\lambda} \label{EigenFunc}
\end{align}
by \eqref{eq12}.
In Theorem \ref{thm:conditionalLimit}, Theorem \ref{thm:conditionalLimitPrevious} is extended to the case where the function $\psi_{-\lambda}$ is replaced by a more general concave function.
In addition, we generalize the result to the case where the boundary $\infty$ is entrance.
A key tool is the $h$-transform in a wide sense studied in \cite{Maeno}, which we review in Section \ref{section:h-transform}.
It allows us to introduce supermartingales corresponding to $M^{[\psi_{-\lambda}]}$ for more general concave functions.

Finally, we note that this paper also contributes through a systematic treatment of general results on one-dimensional diffusions, presented in Section~\ref{section:generalResults}. 
Although parts of this material are classical, we provide new results where needed and, more importantly, a unified analytic treatment based on Feller’s canonical representation of the infinitesimal generator.
We expect that this structured presentation, in addition to its preparatory role, will be of independent value for future work.

\subsection*{Outline of the paper}

In Section \ref{section:preliminary}, we recall several fundamental concepts on one-dimensional diffusions, especially on the characteristic triplet and a representation of the infinitesimal generators established in \cite{Ito-McKean}.
In Section \ref{section:generalResults}, we recall and prepare general results on positive eigenfunctions, the spectra of infinitesimal generators, and exit times.
In Section \ref{section:concaveFunctions}, we introduce a class of concave functions playing an essential role and introduce supermartingales via them.
Section \ref{section:conditinalLimit} is one of the main parts.
We prove the conditional limit theorem which extends Theorem \ref{thm:conditionalLimitPrevious}.
Section \ref{section:absoluteContinuity} is the other main part.
We consider absolute continuity relations of conditional limits on $\cF_{\infty}$ and present sufficient conditions for absolute continuity and singularity. 
We observe examples in Section \ref{section:examples}.

\subsection*{Acknowledgement}

The author is sincerely grateful to Professor Servet Mart\'inez.
The theme of the present paper is based on his suggestion during the author's visit to the Center for Mathematical Modeling in Santiago.
He also provided valuable comments on a draft of the present paper.
The author would also like to express his gratitude to Professor Kouji Yano for reading an early draft and offering many helpful suggestions.
The author is supported by JSPS KAKENHI grant no.\ JP23KJ0236 and JP24K22834.

\section{Preliminary} \label{section:preliminary}

\subsection{Characteristic triplet of one-dimensional diffusions} \label{section:ito-McKean}

As established in \cite{Ito-McKean}, every non-singular one-dimensional diffusion is characterized by a triplet consisting of the scale function, the speed measure, and the killing measure.
Since we work on the framework presented in \cite[Chapter 4,5]{Ito-McKean}, we briefly summarize the theory and necessary results here.
For details, see \cite{Ito-McKean}.
See also \cite{MR3161402} for a treatment from the Dirichlet form theory.
The case without a killing measure is discussed in, e.g., \cite{Feller:thegeneral}, \cite{Ito_essentials}, \cite[Chapter 33]{Kallenberg-third}, and \citet[Chapter V.7]{RogersWilliams2}.

\begin{Rem}
Before getting into the details, we note that, rigorously speaking, the triplet characterization has a slight exception.
In the framework of \cite{Ito-McKean}, the following three behaviors of the process at a boundary are distinguished:
(i) instantaneous killing; 
(ii) absorption with permanent stay; 
(iii) absorption followed by killing after a positive exponential holding time. 
In the third case, an additional parameter is required for the holding time rate.
However, the distinction among these three cases is irrelevant to our main discussion.
Therefore, we identify (ii) and (iii) with (i).
Consequently, the triplet gives a complete characterization.
\end{Rem}

We recall how the triplet arises from a given diffusion.
Let $X = ((X_{t})_{t \geq 0}, (\cF_{t})_{t \geq 0},(\bP_{x})_{x \in I})$ be a non-singular one-dimensional diffusion on $I \cup \{ \partial\}$, where $I$ is an interval with endpoints $\ell$ and $r \ (-\infty \leq \ell < r \leq \infty)$ and $\partial$ is an isolated point which we call the cemetery.
We always consider that any function $f$ on $I \cup \{\partial \}$ satisfies $f(\partial) = 0$.
When $\ell$ or $r$ is entrance in the sense of Feller (see Section~\ref{section:classificationOfTheBoundary}), and if it is not yet contained in $I$, we adjoin it to $I$.
As is classically known, the process $X$ can be extended to a diffusion on the resulting interval $I'$ (see \cite[Problem 3.6.3]{Ito-McKean}).
For notational simplicity, we consider that $X$ is given as a process on $I'$ from the outset, and we write it $I$ instead of $I'$.

Under this modification, the transition semigroup always satisfies the Feller property in the following sense: for every $t > 0$ and $f \in C_{\infty}(I)$ the function
\[
I \ni x \mapsto \mathbb{E}_{x}[f(X_{t})] 
\]
again belongs to $C_{\infty}(I)$ (see \cite[Chapter 4.7]{Ito-McKean}), where $C_{\infty}(I)$ denotes the set of continuous functions on $I$ vanishing at infinity.
By the Hille-Yosida theorem, the infinitesimal generator $L$ on $C_{\infty}(I)$ defined by
\[
Lf(x) := \lim_{t \to 0}\frac{\bE_{x}[f(X_{t})] - f(x)}{t} 
\]
gives a densely defined closed operator on $C_{\infty}(I)$.
We denote the domain of $L$ by $\cD(L)$.

In \cite[Chapter 3,4,5]{Ito-McKean}, it has been established that there exist Radon measures $m$ with full support and $k$, and a strictly increasing and continuous function $s$ on $I$,  and the generator $L$ satisfies for every $f \in \cD(L)$
\begin{align}
\int_{a}^{b}Lf(x)dm(x) = \frac{d^{+}f}{ds}(b) -  \frac{d^{+}f}{ds}(a) - \int_{a}^{b}f(x)dk(x) \quad (a,b \in (\ell,r)). \label{eq65}
\end{align}
This is symbolically represented as
\[
Lf(x) = \frac{\frac{d^{+}f}{ds}(dx) - f(x)dk(x)}{dm(x)} \quad (x \in (\ell,r))
\]
and can be understood that the infinitesimal generator of $X$ is represented by a second-order differential operator.
If an endpoint $\Delta \in \{\ell,r \} \cap I$ is present, we also have equations written by $m,s,k$ for $Lf(\Delta)$ to satisfy, though we omit the detail here (see \cite[Chapter 4.7]{Ito-McKean}).
In summary, the generator $L$, and consequently the process $X$, is completely characterized by the \emph{speed measure} $m$, the \emph{scale function} $s$, and the \emph{killing measure} $k$.
We sometimes write $L$ as $L^{(m,s,k)}$.

We call a diffusion associated with a triplet $(m,s,k)$ an $(m,s,k)$-diffusion.
In particular, if $k=0$, we sometimes call it a $\frac{d}{dm}\frac{d^{+}}{ds}$-diffusion.
Moreover, when the scale function is identity (i.e., $s(x) = x$), we say that the corresponding diffusion is natural-scale.

\subsection{Feller's classification of the boundary} \label{section:classificationOfTheBoundary}

Let $L$ be an infinitesimal generator of a diffusion, and let $(m,s,k)$ be the corresponding triplet.
Fix $c \in (\ell,r)$ arbitrarily and set $dj(x) := dm(x) + dk(x)$.
For the boundary point $\Delta = \ell$ or $r$, define
\begin{align}
	I(\Delta) := \int_{\Delta}^{c}ds(x)\int_{\Delta}^{x}dj(y), \quad J(\Delta) := \int_{\Delta}^{c}dj(x)\int_{\Delta}^{x}ds(y). \label{}
\end{align}
Throughout the present paper, we interpret $\int_{a}^{b} := \int_{(a,b]}$ and $\int_{b}^{a} := -\int_{a}^{b}$ for $a < b$.
According to Feller's classification, the boundary $\Delta$ is called 
\begin{align}
	\begin{cases}
		\text{regular} &\text{when}\  I(\Delta) < \infty \ \text{and} \  J(\Delta) < \infty \\
		\text{exit} &\text{when} \ I(\Delta) = \infty \ \text{and} \  J(\Delta) < \infty \\
		\text{entrance} &\text{when} \ I(\Delta) < \infty \ \text{and} \  J(\Delta) = \infty \\
		\text{natural} &\text{when} \ I(\Delta) = \infty \ \text{and} \  J(\Delta) = \infty
	\end{cases}
	.
	\label{}
\end{align}
We also say that $\Delta$ is regular for $L$, or the like.

It is well known that this classification is closely related to the behavior of $(m,s,k)$-diffusion (see e.g., \cite[Chapter V.7]{RogersWilliams2}, \citet[Chapter 5.16]{Ito_essentials}, and \citet[Theorem 33.12]{Kallenberg-third} for the case of $k=0$).

Now we fix our setting.
We denote by $\Omega$ the space of functions $\omega$ from $[0,\infty)$ to $(0,\infty) \cup \{\partial
\}$ such that the function $[0,\zeta(\omega)) \ni t \mapsto \omega(t)$ is continuous and $\omega(t) = \partial \ (t \geq \zeta (\omega))$ for $\zeta(\omega) := \inf \{ t \geq 0 \mid \omega (t) = \partial \}$.
We also denote by $\cF_{t}$ the right-continuous induced filtration, i.e., $\cF_{t} := \bigcap_{s > t} \sigma (X_{u}, u \leq s)$.
Set $\cF_{\infty} := \sigma \left( \bigcup_{t > 0} \cF_{t} \right)$.
For simplicity, we always work on the measurable space $(\Omega,\cF_{\infty})$.
In the remainder of this paper, we fix a non-singular and natural-scale $\frac{d}{dm}\frac{d^{+}}{dx}$-diffusion
\[
X = ((X_{t})_{t \geq 0},(\bP_{x})_{x \geq 0},(\cF_{t})_{t \geq 0})
\]
on the half-line $(0,\infty)$ only killed at $T_{0} := \lim_{x \to 0}T_{x}$, where $T_{x} := \inf \{ t > 0 \mid X_{t} = x \}$, i.e., $\zeta = T_{0}$.
We denote the transition semigroup by $p_{t}$; $p_{t}f(x) := \bE_{x}[f(X_{t})]$.
Since we are interested in conditionings to avoid zero, we assume the boundary $0$ is regular or exit in the sense of Feller.
Furthermore, we assume that the boundary $\infty$ is entrance or natural, which does not significantly restrict our problem (see Remark \ref{rem:setting}).

Note that since $X$ is natural-scale and the state space $(0,\infty)$ is unbounded, the process $X$ is certainly absorbed in zero: $\bP_{x}[T_{0} < \infty] = 1 \ (x > 0)$ (see e.g., \citet[Theorem 33.15]{Kallenberg-third}).

\begin{Rem} \label{rem:setting}
In our setting, the state space is assumed to be unbounded above and the boundary $\infty$ to be entrance or natural.
For the main subject of this paper, this restriction essentially excludes no meaningful cases, since all other cases are either trivial or can be discussed similarly to the case addressed in this paper.

Note that when the state space is unbounded above, the boundary $\infty$ is automatically either entrance or natural.
Hence, it suffices to verify that when the state space is a bounded interval $[0,\ell)$ or $[0,\ell]$ with $\ell \in (0,\infty)$, nothing interesting occurs.

When $\ell$ is regular absorbing, regular elastic, or exit, the process $X$ is absorbed at $\ell$ with positive probability, and hence $\mathbb{P}_{x}[T_{0} = \infty] > 0$.
In this case, for any clock $\tau_{\lambda}$, we obviously have the following conditional limit to avoid zero:
\[
\lim_{\lambda \to \infty} \mathbb{P}_{x}[A \mid T_{0} > \tau_{\lambda}] 
= \mathbb{P}_{x}[A \mid T_{0} = \infty], 
\quad (t \geq 0,\ A \in \cF_{t}).
\]
When $\ell$ is natural, then $X_{t} \to \ell$ as $t \to \infty$ with positive probability (see, e.g., \cite[Theorem~33.15]{Kallenberg-third}), and the situation is trivial for the same reason.
The boundary $\ell$ cannot be entrance when $\ell < \infty$.

Finally, the case where $\ell$ is regular reflecting can be handled by the same arguments with the unbounded interval case we mainly consider.
We therefore omit it for brevity.
\end{Rem}

\subsection{Local generator}

As described in Section \ref{section:ito-McKean}, the transition semigroup of $X$ satisfies the Feller property, i.e., $p_{t}$ maps $C_{\infty}(I)$ to itself.
From \eqref{eq65}, the generator $L$ of $X$ is represented by $\frac{d}{dm}\frac{d^{+}}{dx}$ on $(0,\infty)$:
\[
Lf(x) = \frac{d}{dm}\frac{d^{+}}{dx}f(x) \quad (f \in \cD(L), \ x \in (0,\infty)),
\]
where we interpret $\frac{d}{dm}$ as the Radon-Nikodym derivation.

It is convenient to consider the differential operator independently.
Let $\cD(\cL)$ be the set of continuous functions $u:(0,\infty) \to \bR$ such that there exists a continuous function $v:(0,\infty) \to \bR$ and
\begin{align}
	\frac{d}{dm}\frac{d^{+}}{dx}u(x) = v(x) \quad (x \in (0,\infty)), \label{eq43}
\end{align}
and we define $\cL u := v$.
We call $\cL$ the \textit{local generator} on $(0,\infty)$ (see also \cite[Chapter 5.3]{Ito_essentials}).

\subsection{Transition and resolvent density}

We recall basic facts on the transition and resolvent density of one-dimensional diffusions.
See e.g., \cite{McKean:elementary} and \cite[Chapter 4.11]{Ito-McKean} for details.

The process $X$ has a symmetric transition density with several desirable properties.
There is a continuous function $p(t,x,y)$ $(t > 0, x,y \in (0,\infty))$ such that $p(t,x,y) = p(t,y,x)$ and
\[
p_{t}f(x) = \int_{0}^{\infty}p(t,x,y)f(y)dm(y) \quad (t > 0, \ x \in (0,\infty))
\]
for every non-negative measurable function $f$.
We also have suitable boundary asymptotics of $p(t,x,y)$ depending on the boundary classification though we omit here.

The existence of the transition density implies that of the resolvent density:
\[
R^{(q)}f(x) := \int_{0}^{\infty}\mathrm{e}^{-qt}p_{t}f(x)dt = \int_{0}^{\infty}r^{(q)}(x,y)f(y)dm(y)
\]
for 
\[
r^{(q)}(x,y) := \int_{0}^{\infty}\mathrm{e}^{-qt}p(t,x,y)dt \ (q \geq 0, \ x,y \in (0,\infty)).
\]

By the classical theory of ordinary differential equations, the density $r^{(q)}$ can be expressed in terms of positive increasing and decreasing eigenfunctions of $\cL$.
Let us recall $\psi_{q}$ introduced in \eqref{eq12}.
From the definition of $\cL$, it satisfies $\cL \psi_{q} = q\psi_{q}$.
In this sense, we say that $\psi_{q}$ is a $q$-eigenfunction of $\cL$.
For $q \geq 0$, it is positive and increasing.
More strongly, we can easily check the following (see e.g., \cite[Chapter 5.12]{Ito_essentials}):
\[
x \leq \psi_{q}(x) \leq x \exp \left( q\int_{0}^{x}ydm(y) \right) \quad (x > 0)
\]
and
\[
1 \leq \psi_{q}^{+}(x) \leq  \exp \left( q\int_{0}^{x}ydm(y) \right) \quad (x > 0).
\]
Define
\begin{align}
	g_{q}(x) := \psi_{q}(x)\int_{x}^{\infty}\frac{dy}{\psi_{q}(y)^{2}} \ (q \geq 0, \ x \in (0,\infty)). \label{g_q}
\end{align}
The function $u = g_{q}$ is a positive decreasing solution of $\cL$.
More specifically, we have the following properties.
\begin{Prop} \label{prop:g_qProperty}
	For $q \geq 0$, the integral in the RHS of \eqref{g_q} is finite, and the following hold:
	\begin{enumerate}
		\item $\cL g_{q} = q g_{q}$.
		\item $g_{q}(x) > 0$ and $g_{q}^{+}(x) < 0 \ (x > 0)$.
		\item $\lim_{x \to 0} g_{q}(x) = 1$ and $ \lim_{x \to \infty}g_{q}^{+}(x) = 0$.
		\item The Wronskian of $g_{q}$ and $\psi_{q}$ is one:
		\[
		g_{q}(x)\psi_{q}^{+}(x) - g_{q}^{+}(x)\psi_{q}(x) = 1 \ (x \in (0,\infty)).
		\]
	\end{enumerate}
\end{Prop}

Although these properties are classically known, we present a proof for completeness.

\begin{proof}[Proof of Proposition \ref{prop:g_qProperty}]
	The integral $\int_{x}^{\infty}dy/\psi_{q}(y)^{2}$ is finite by the inequality $\psi_{q}(x) \geq x$.
	The positivity $g_{q}(x) > 0 \ (x > 0)$ is obvious.
	Since $\psi_{q}$ is convex, we have
	\begin{align}
		g^{+}_{q}(x) = \psi_{q}^{+}(x) \int_{x}^{\infty}\frac{dy}{\psi_{q}(y)^{2}} - \frac{1}{\psi_{q}(x)} < \int_{x}^{\infty}\frac{\psi_{q}^{+}(y)}{\psi_{q}(y)^{2}}dy - \frac{1}{\psi_{q}(x)} = 0. \label{eq15}
	\end{align} 
	Thus, (ii) is shown.
	Since $d\psi_{q}^{+} = q \psi_{q}dm$, it follows
	\begin{align}
		d(g^{+}_{q}(x)) &= d\left( \psi_{q}^{+}(x)\int_{x}^{\infty}\frac{dy}{\psi_{q}(y)^{2}} - \frac{1}{\psi_{q}(x)} \right) \label{} \\
		&= q g_{q}(x) dm - \frac{\psi_{q}^{+}(x-)}{\psi_{q}^{2}(x)}dx + \frac{\psi_{q}^{+}(x)}{\psi_{q}^{2}(x)}dx \label{} \\
		&= q g_{q} dm, \label{}
	\end{align}
	for $\psi_{q}^{+}(x-) := \lim_{y \up x}\psi_{q}^{+}(y)$.
	Note that $\psi_{q}^{+}(x) = \psi_{q}^{+}(x-)$ except at points $x$ with $m\{x\} > 0$. 
	Since such points are at most countable, they form a null set for $dx$ and may be ignored.
	Thus, we obtain $\cL g_{q} = q g_{q}$.
	We show (iii).
	The equality $\lim_{x \to 0}g_{q}(x) = 1$ follows from $\lim_{x \to 0}\psi_{q}(x)/x = 1$.
	From \eqref{eq15}, we have $|g_{q}^{+}(x)| \leq 1/ \psi_{q}(x) \to 0 \ (x \to \infty)$.
	Property (iv) directly follows from the first equality of \eqref{eq15}.
	The proof is complete.
\end{proof}

The resolvent density can be written as follows (see e.g., \cite[Chapter 5.14, 5.15]{Ito_essentials}):
\begin{align}
r^{(q)}(x,y) := 
\begin{cases}
	\psi_{q}(x)g_{q}(y) & (x \leq y) \\
	r^{(q)}(y,x) & (x > y)
\end{cases}
. \label{defOfResolventDensity}
\end{align}

Since the domain $\cD(L)$ satisfies $\cD(L) = R^{(q)}(C_{\infty}(I)) \ (q > 0)$, we can easily identify it through the boundary behavior of $\psi_{q}$ and $g_{q}$ (see e.g., \cite[p.136]{Ito-McKean} and \cite[Chapter 5]{Ito_essentials}).
When $\infty$ is entrance,
\begin{align}
  \cD(L) =
    \left\{ u \in \cD(\cL) \cap C_{\infty}(0,\infty] \,\middle|\,
	\lim_{x \to 0}\cL u(x) = 0, \ 
	\lim_{x \to \infty} u^{+}(x) = 0, \ 
	\lim_{x \to \infty} \cL u(x) \in \bR
	\right\},
\end{align}
and when $\infty$ is natural.
\begin{align}
	\cD(L) = 
	\left\{ u \in \cD(\cL) \cap C_{\infty}(0,\infty) \,\middle|\,
    \lim_{x \to 0} \cL u(x) = 0, \ 
    \lim_{x \to \infty}\cL u(x) = 0
    \right\}
	.
\end{align}

\subsection{Generators on $L^{p}$-spaces}

For $p \in [1,\infty)$, let $L^{p}(m)$ be the set of functions $f$ on $(0,\infty)$ with $||f||_{p} := \int_{0}^{\infty}|f|^{p}dm < \infty$, modulo the ideal of null functions.
Let $L^{\infty}(m)$ be the set of essentially bounded functions $f$ on $(0,\infty)$ with respect to $dm$, modulo the ideal of null functions, and we write $||f||_{\infty} := \mathrm{ess sup}_{x}|f(x)|$.
Hereafter, we do not distinguish notationally between an element in $L^{p}(m)$ and its version.

As is classically known, the transition semigroup is symmetric on $L^{2}(m)$:
\[
\int_{0}^{\infty}p_{t}f(x)g(x)dm(x) = \int_{0}^{\infty}f(x)p_{t}g(x)dm(x) \quad (f,g \in L^{2}(m)),
\]
which can also be confirmed through the symmetry of the transition density $p(t,x,y)$.
This implies that the transition semigroup $p_{t}$ can be extended to a positive contraction semigroup $T_{t}$ on $L^{p}(m) \ (p \in [1,\infty])$, and it is strongly continuous for $p \in [1,\infty)$ (see e.g., \cite[Theorem 1.4.1]{DaviesHeatKernel}). For $p \in [1,\infty)$, we denote the infinitesimal generator by $L_{p}$ and the domain by $\cD(L_{p})$.
For $p = \infty$, we define $L_{\infty}$ as the adjoint operator of $L_{1}$.
We denote by $\sigma(L_{p})$ the spectrum of $L_{p}$.

By similar considerations as for the domain $\cD(L)$, the domain $\cD(L_{p}) \ (p \in [1,\infty))$ can be identified.
For measures $\mu$ and $\nu$, we write $\mu \ll \nu$ when $\mu$ is absolutely continuous with respect to $\nu$.
Set
\[
\cD(\cL') := \{ f:(0,\infty) \to \bR \mid df \ll dx,\  df' \ll dm \}
\]
and define $\cL' f = \frac{d}{dm}f'$.
We see that when $\infty$ is entrance,
\begin{align}
\cD(L_{p})
= \left\{ f \in L^{p}(m)\ \middle|\ 
\begin{array}{l}
f \in \cD(\cL'),\\
\int |\cL'f|^{p} dm < \infty,\\
\lim_{x \to 0}\cL' f(x)=0,\\
\lim_{x \to \infty} f'(x)=0
\end{array}
\right\}, \label{LpDomainEntrance}
\end{align}
and when $\infty$ is natural,
\begin{align}
\cD(L_{p})
= \left\{ f \in L^{p}(m)\ \middle|\ 
\begin{array}{l}
f \in \cD(\cL'),\\
\int |\cL'f|^{p} dm < \infty,\\
\lim_{x \to 0}\cL' f(x)=0
\end{array}
\right\}. \label{LpDomainNatural}
\end{align}

In our discussion, the bottom of the spectrum of $-L_{2}$
\[
\lambda_{0} := -\sup \sigma(L_{2})
\]
is particularly important.
A simple characterization of the positivity of $\lambda_{0}$ is known.

\begin{Thm}[{\citet[Appendix, Theorem 3 (ii)]{KotaniWatanabe}}]
	We have $\lambda_{0} > 0$ if and only if
	\begin{align}
	m(1,\infty) < \infty \quad \text{and} \quad \limsup_{x \to \infty} x m(x,\infty) < \infty. \label{eq46}
	\end{align}
\end{Thm}

\begin{Rem}
	When $\infty$ is entrance, the condition \eqref{eq46} always holds.
\end{Rem}

\section{Some general results on one-dimensional diffusions} \label{section:generalResults}

In this section, we present several general results on one-dimensional diffusions required for our purposes. 
While parts of them are classical, we also include some new and generalized results.

\subsection{Positive eigenfunctions} \label{section:QSD}

We recall basic results on positive eigenfunctions of the local generator.
Its existence is closely related to the positivity of $\lambda_{0}$.

The following proposition is fundamental for the monotonicity and positivity of eigenfunctions, which is a consequence of the classical theory of second-order ordinary differential operators.

\begin{Prop} \label{prop:monotonicityOfPsi}
	The following holds:
	\begin{align}
		\lambda_{0} = \sup\{ \lambda \in \bR \mid \psi_{-\lambda}(x) > 0 \ (x > 0) \} 
		= \sup\{ \lambda \in \bR \mid \psi_{-\lambda}^{+}(x) > 0 \ (x > 0) \}. \label{eq47}
	\end{align}
	More precisely, we have
	\begin{align}
	0 < \psi_{q'}(x) < \psi_{q}(x) \quad \text{and} \quad 0 < \psi^{+}_{q'}(x) < \psi^{+}_{q}(x) \quad (x > 0,\ -\lambda_{0} \leq q' < q). \label{eq48}
	\end{align}
	In particular, for $q \in \bR$ it follows that
	\begin{align}
	\psi_{q}(x) > 0 \ \text{for every $x > 0$} \quad \text{if and only if} \quad q \geq -\lambda_{0}, \label{eq49}
	\end{align}
	and the function
	\begin{align}
		[-\lambda_{0},\infty) \times (0,\infty) \ni (q,x) \longmapsto \psi_{q}(x) \label{eq50}
	\end{align}
	is increasing in $q$ and $x$.
\end{Prop}

\begin{Rem}
	Parts of this proposition is shown in \cite[Lemma 1]{Ratesofdecay} when $\infty$ is entrance.
\end{Rem}

\begin{proof}[Proof of Proposition \ref{prop:monotonicityOfPsi}]
	We first check that \eqref{eq48} follows from \eqref{eq47}, although the derivation is no more than the classical Sturm's comparison argument (see e.g., \cite[Chapter 8]{Coddington}).

	Assume \eqref{eq47} holds.
	Let $q \in (-\lambda_{0},\infty)$.
	Take $\lambda \in (-q,\lambda_{0}]$ such that $\psi_{-\lambda}(x) > 0 \ (x > 0)$.
	We have
	\begin{align}
		\frac{d^{+}}{dx}\frac{\psi_{q}(x)}{\psi_{-\lambda}(x)} 
		&= \frac{\psi_{q}^{+}(x)\psi_{-\lambda}(x) - \psi_{q}(x)\psi_{-\lambda}^{+}(x)}{\psi_{-\lambda}(x)^{2}} \label{} \\
		&= \frac{(\lambda + q) \int_{0}^{x}\psi_{q}(y)\psi_{-\lambda}(y)dm(y)}{\psi_{-\lambda}(x)^{2}}. \label{eq51}
	\end{align}
	We prove by contradiction that $\psi_{q}(x) > 0 \ (x > 0)$.
	Suppose there exists $x_{0} > 0$ such that $\psi_{q}(x_{0}) = 0$.
	Since $\psi^{+}_{q}(0) > 0$, we may assume $x_{0} = \min \{ x > 0 \mid \psi_{q}(x) = 0\} > 0$.
	From \eqref{eq51} it follows for every $c \in (0,x_{0})$
	\[
	0 > - \frac{\psi_{q}(c)}{\psi_{-\lambda}(c)} = \int_{c}^{x_{0}}\frac{d^{+}}{dx}\frac{\psi_{q}(x)}{\psi_{-\lambda}(x)}dx > 0,
	\]
	and this is contradiction.
	Thus, we have $\psi_{q}(x) > 0 \ (x > 0)$.
	Integrating \eqref{eq51} from $0$ to $x$, we obtain $\psi_{q}(x) > \psi_{-\lambda}(x) \ (x > 0)$ since $\lim_{x \to 0}\psi_{q}(x) / \psi_{-\lambda}(x) = 1$.
	This shows the first relation in \eqref{eq48} for $-\lambda_{0} < q' < q$.
	A similar argument for $\psi_{q}^{+}$ instead of $\psi_{q}$ gives the second relation in \eqref{eq48} for $-\lambda_{0} < q' < q$.
	We check \eqref{eq48} for $q' = -\lambda_{0}$.
	From the continuity of $\bC \ni q \mapsto \psi_{q}(x)$, we have
	\[
	\psi_{-\lambda_{0}}(x) \geq 0 \quad \text{and} \quad \psi^{+}_{-\lambda_{0}}(x) \geq 0 \quad \text{for every $x > 0$},
	\]
	which obviously implies $\psi_{-\lambda_{0}}(x) > 0 \ (x > 0)$.
	Thus, \eqref{eq47} entails \eqref{eq48}.

	To prove \eqref{eq47}, we first recall a consequence of the classical Sturm-Liouville theory.
	Fix $r > 0$.
	Define $L^{2}((0,r),dm)$ as the set of functions $g$ on $(0,r)$ with $\int_{0}^{r}|g|^{2}dm < \infty$, modulo null functions.
	On this space, let $L^{D,r}$ and $L^{N,r}$ denote the non-positive definite self-adjoint operators $\frac{d}{dm^{(r)}}\frac{d^{+}}{dx}$ on $L^{2}((0,r),dm)$, where $dm^{(r)} := dm|_{(0,r)}$, equipped with the Dirichlet and Neumann boundary condition at $r$, respectively.
	For both operators, we impose the Dirichlet boundary condition at $0$.
	Set 
	\[
	\lambda_{0}^{D,r} := -\sup \sigma(L^{D,r}) \quad \text{and} \quad \lambda_{0}^{N,r} := -\sup \sigma(L^{N,r}).
	\]
	It is classically known that 
	\[
	\lambda_{0}^{D,r} = \sup \{ \lambda \geq 0 \mid \psi_{-\lambda}(x) > 0 \ (x \in (0,r)) \}
	\]
	and
	\[
	\lambda_{0}^{N,r} = \sup \{ \lambda \geq 0 \mid \psi^{+}_{-\lambda}(x) > 0 \ (x \in (0,r)) \},
	\]
	which implies $\lambda_{0}^{N,r} < \lambda_{0}^{D,r}$.
	It is elementary to check that 
	\[
	\lambda_{0}^{D} := \lim_{r \to \infty} \lambda_{0}^{D,r} = \sup \{ \lambda \geq 0 \mid \psi_{-\lambda}(x) > 0 \ (x \in (0,\infty)) \}
	\]
	and
	\[
	\lambda_{0}^{N} := \lim_{r \to \infty}\lambda_{0}^{N,r} = \sup \{ \lambda \geq 0 \mid \psi^{+}_{-\lambda}(x) > 0 \ (x \in (0,\infty)) \}.
	\]
	To complete the proof, it suffices to show $\lambda_{0}^{D} = \lambda_{0}^{N} = \lambda_{0}$.

	Recall the variational characterization of $\lambda_{0}$:
	\[
	\lambda_{0} = \inf_{f \in \cD(L_{2})}\frac{\langle -L_{2}f,f \rangle}{||f||_{2}^{2}},
	\]
	where $\langle f, g \rangle := \int_{0}^{\infty}f(x)g(x)dm(x)$ and $||f||_{2} := \sqrt{\langle f, f \rangle}$.
	We can easily verify that any element $f \in \cD(L^{D,r})$ can be regarded as that in $\cD(L_{2})$ by extending it as $f(x+r) := f(r) \ (x > 0)$, regardless of the boundary classification of $\infty$, which means $\cD(L^{D,r}) \subset \cD(L_{2})$.
	Then, again from the variational formula, we see that $\lambda_{0} \leq \lambda_{0}^{D,r}$, i.e., $\lambda_{0} \leq \lambda_{0}^{D}$.
	Let $R^{D,r}_{1}$ be the $1$-resolvent of $L^{D,r}$.
	We can routinely show that $\bigcup_{r > 0}R^{D,r}_{1}(L^{2}((0,r),dm))$ is a core of $L_{2}$, which readily implies $\lambda_{0}^{D} \leq \lambda_{0}$ by the variational formula.
	The proof is complete.
\end{proof}

Proposition \ref{prop:monotonicityOfPsi} shows that for $\lambda \in (0,\lambda_{0}]$, the function $\psi_{-\lambda}(x) \ (x > 0)$ is positive, increasing, and concave.
The behavior of these functions at $x = \infty$ differs depending on the boundary classification and whether $\lambda = \lambda_{0}$ or not.  

\begin{Prop} \label{prop2-8}
	Suppose $\lambda_{0} > 0$.
	The following hold:
	\begin{enumerate}
		\item When the boundary $\infty$ is entrance, we have for $\lambda \in (0,\lambda_{0})$
		\[
		\lim_{x \to \infty} \psi_{-\lambda}(x) = \infty, \quad \lim_{x \to \infty} \psi^{+}_{-\lambda}(x) \in (0,\infty), \quad \text{and} \quad \lambda\int_{0}^{\infty}\psi_{-\lambda}(x)dm(x) < 1,
		\]
		and
		\[
		\sup_{x > 0} \psi_{-\lambda_{0}}(x) < \infty, \quad \lim_{x \to \infty} \psi^{+}_{-\lambda_{0}}(x) = 0, \quad \text{and} \quad \lambda\int_{0}^{\infty}\psi_{-\lambda_{0}}(x)dm(x) = 1.
		\]
		\item When the boundary $\infty$ is natural, we have for $\lambda \in (0,\lambda_{0}]$
		\[
		\lim_{x \to \infty} \psi_{-\lambda}(x) = \infty, \quad \lim_{x \to \infty} \psi^{+}_{-\lambda}(x) = 0, \quad \text{and} \quad \lambda\int_{0}^{\infty}\psi_{-\lambda}(x)dm(x) = 1.
		\]
	\end{enumerate}
	In addition, when $\lim_{x \to \infty} \psi^{+}_{-\lambda}(x) = 0$ for $\lambda \in (0,\lambda_{0}]$, we have the following representation: 
	\begin{align}
		\psi_{-\lambda}(x) = \lambda \int_{0}^{x}dy\int_{y}^{\infty}\psi_{-\lambda}(z)dm(z) \quad (x > 0). \label{eq21}
	\end{align}
\end{Prop}

\begin{proof}
	Let $\lambda \in (0,\lambda_{0}]$.
	From Proposition \ref{prop:monotonicityOfPsi}, it follows $\psi_{-\lambda}(x), \psi_{-\lambda}^{+}(x) > 0 \ (x > 0)$.
	Since
	\[
	\psi_{-\lambda}^{+}(x) = 1 - \lambda \int_{0}^{x}\psi_{-\lambda}(y)dm(y),
	\]
	we have  $\int_{0}^{\infty}\psi_{-\lambda}(y)dm(y) < \infty$ and $\psi_{-\lambda}^{+}(\infty) := \lim_{x \to \infty}\psi_{-\lambda}^{+}(x) \in [0,1)$.
	Suppose $\psi_{-\lambda}^{+}(\infty) = 0$.
	It follows that $\lambda\int_{0}^{\infty}\psi_{-\lambda}(x)dm(x) = 1$ and
	\begin{align}
		\psi_{-\lambda}(x) &= x - \lambda \int_{0}^{x}dy\int_{0}^{y}\psi_{-\lambda}(z)dm(z) \nonumber \\
		&= \lambda \int_{0}^{x}dy\int_{0}^{\infty}\psi_{-\lambda}(z)dm(z)  - \lambda \int_{0}^{x}dy\int_{0}^{y}\psi_{-\lambda}(z)dm(z) \nonumber \\
		&= \lambda \int_{0}^{x}dy\int_{y}^{\infty}\psi_{-\lambda}(z)dm(z), \nonumber
	\end{align}
	and \eqref{eq21} holds.

	Suppose $\infty$ is natural.
	If $\psi_{-\lambda}^{+}(\infty) > 0$, it follows $\psi_{-\lambda}^{+}(x) \geq \psi_{-\lambda}^{+}(\infty)$, and $\psi_{-\lambda}(x) \geq \psi^{+}_{-\lambda}(\infty)x \ (x > 0)$.
	This contradicts to $\int_{0}^{\infty}\psi_{-\lambda}(y)dm(y) < \infty$ since $\int_{1}^{\infty}x dm(x) = \infty$.
	Thus, we have $\psi_{-\lambda}^{+}(\infty) = 0$ and $\lambda\int_{0}^{\infty}\psi_{-\lambda_{0}}(x)dm(x) = 1$.
	We see that $\lim_{x \to \infty}\psi_{-\lambda}(x) = \infty$ by
	\begin{align}
		\lim_{x \to \infty}\psi_{-\lambda}(x) &= \lambda\int_{0}^{\infty}dy\int_{y}^{\infty}\psi_{-\lambda}(z)dm(z) \label{} \\
		&= \lambda \int_{0}^{\infty}z\psi_{-\lambda}(z)dm(z) \label{} \\
		&\geq \psi_{-\lambda}(1) \int_{1}^{\infty}zdm(z) = \infty. \label{}
	\end{align}

	Suppose $\infty$ is entrance.
	It is classically known that $R^{(q)} \ (q > 0)$ is a trace-class operator on $L^{2}(m)$, and thus the spectrum consists of eigenvalues.
	Since $-\lambda_{0} \in \sigma(L_{2})$, we have $\psi_{-\lambda_{0}}^{+}(\infty) = 0$ from \eqref{LpDomainEntrance}.
	From the assumption of the entrance boundary, we can easily confirm that the function
	\[
	q \mapsto 1 - q \int_{0}^{\infty}\psi_{-q}(y)dm(y) \ (= \psi_{-q}^{+}(\infty))
	\]
	is entire.
	Since the zeros of the analytic function $\psi_{-\lambda}^{+}(\infty)$ are discrete and $(0,\lambda_{0}] \ni \lambda \mapsto \psi_{-\lambda}(\infty)$ is non-increasing, it follows that $\psi_{-\lambda}^{+}(\infty) > 0 \ (\lambda \in (0,\lambda_{0}))$.
	This obviously implies $\lim_{x \to \infty} \psi_{-\lambda}(x) = \infty$ and $\lambda\int_{0}^{\infty}\psi_{-\lambda}(x)dm(x) < 1$.

	Since $\psi_{-\lambda_{0}}$ is a $(-\lambda_{0})$-eigenfunction of $L_{2}$, $\mathrm{e}^{\lambda_{0}t}\psi_{-\lambda_{0}}(X_{t})$ is a square-integrable martingale with respect to $\bP_{x}$ for every $x > 0$.
	From \cite[Proposition 2.4, Remark 2.3]{entranceBdry}, there exists $r > 0$ such that $\sup_{x \geq r}\bE_{x}[\mathrm{e}^{\lambda_{0} T_{r}}] < \infty$.
	By the optional stopping theorem, we have for every $x \geq r$
	\[
	\psi_{-\lambda_{0}}(x) = \bE_{x}[\mathrm{e}^{\lambda_{0} T_{r}}\psi_{-\lambda_{0}}(X_{T_{r}})] \leq \psi_{-\lambda_{0}}(r) \sup_{y \geq r}\bE_{y}[\mathrm{e}^{\lambda_{0} T_{r}}].
	\]
	Therefore, the function $\psi_{-\lambda_{0}}$ is bounded.
	The proof is complete.
\end{proof}

\begin{Rem}
	\begin{enumerate}
		\item The boundedness of an eigenfunction in the entrance case is a special case of the general result of \cite[Theorem 5]{TakedaTightness}, and our proof is based on the proof.
		\item We can also show by the same argument that every $L^2(m)$-eigenfunction is bounded when $\infty$ is entrance.
	\end{enumerate}
\end{Rem}

We can also characterize the square-integrability of $\psi_{-\lambda}$.

\begin{Prop} \label{prop:square-integrabilityOfPsi}
	The following hold:
	\begin{enumerate}
		\item When $\int_{1}^{\infty}x^{2}dm(x) < \infty$, we have $\int_{0}^{\infty}|\psi_{q}(x)|^{2}dm(x) < \infty$ for every $q \in \bC$.
		\item When $\int_{1}^{\infty}x^{2}dm(x) = \infty$ and $\lambda_{0} > 0$, we have $\int_{0}^{\infty}\psi_{-\lambda}(x)^{2}dm(x) = \infty$ for every $\lambda \in (0,\lambda_{0})$.
	\end{enumerate}
\end{Prop}

\begin{proof}
	Let $\int_{1}^{\infty}x^{2}dm(x) < \infty$ and $q \in \bC$.
	Part (i) readily follows from the following estimate derived by \eqref{eq12}:
	\[
	|\psi_{q}(x)| \leq x \mathrm{e}^{|q|\int_{0}^{x}ydm(y)}.
	\]

	Suppose $\int_{1}^{\infty}x^{2}dm(x) = \infty$ and $\lambda_{0} > 0$.
	Let $\lambda \in (0,\lambda_{0})$.
	When $\infty$ is entrance, Proposition \ref{prop2-8} (i) implies that $0 < \psi^{+}_{-\lambda}(\infty)x \leq \psi_{-\lambda}(x) \ (x > 0)$, and hence $\int_{0}^{\infty}\psi_{-\lambda}(x)^{2}dm(x) = \infty$.
	Let $\infty$ be natural.
	If $\psi_{-\lambda}$ is square-integrable, it follows from \eqref{LpDomainNatural} that $-\lambda \in \sigma(L_{2})$, and that contradicts to the definition of $\lambda_{0}$.
	The proof is complete.
\end{proof}

\begin{Rem}
	Part (i) of Proposition \ref{prop:square-integrabilityOfPsi} can be shown through the limit-circle/limit-point classification at $\infty$ (see e.g., \cite[Chapter 9.2]{Coddington}).
\end{Rem}

From Proposition \ref{prop2-8}, we can determine when $\psi_{-\lambda}$ is an eigenfunction of $L_{1}$.

\begin{Prop} \label{prop:eigenvaluesOfGenerator}
	Let $p \in [1,\infty]$ and let $\sigma_{\mathrm{eigen}}(L_{p})$ be the set of eigenvalues of $L_{p}$.
	The following hold:
	\begin{enumerate}
		\item When the boundary $\infty$ is entrance, we have $-\lambda_{0} \in \sigma_{\mathrm{eigen}}(L_{p})$ and $\sigma(L_{p}) \subset \{ z \in \bC \mid \Re z \leq -\lambda_{0} \}$.
		\item When the boundary $\infty$ is natural, we have $[-\lambda_{0},0) \in \sigma_{\mathrm{eigen}}(L_{1})$.
	\end{enumerate}
\end{Prop}

\begin{proof}
	Part (ii) readily follows from \eqref{LpDomainNatural} and Proposition \ref{prop2-8} (ii).

	Suppose $\infty$ is entrance.
	We see that $-\lambda_{0} \in \sigma_{p}(L_{1})$ from \eqref{LpDomainEntrance} and Proposition \ref{prop2-8} (i).
	From \cite[Section 2 (v)]{TakedaTightness}, the process $X$ is in class (T) (see \cite[p.489]{TakedaTightness}).
	From \cite[Theorem 2]{TakedaTightness}, we have for $p \in [1,\infty]$
	\[
	\lambda_{0} = -\lim_{t \to \infty}\frac{1}{t}\log ||T_{t}||_{p \to p},
	\]
	where $||\cdot||_{p \to p}$ denotes the operator norm on $L^{p}(m)$.
	This, with the Hille-Yosida theorem, shows $\sigma(L_{p}) \subset \{ z \in \bC \mid \Re z \leq -\lambda_{0} \}$ for $p \in [1,\infty)$.
	The case of $p = \infty$ follows from the fact that $\sigma(L_{\infty})$ is the complex conjugate of $\sigma(L_{1})$.
\end{proof}

\subsection{Laplace transforms of exit times} \label{section:LTofHittingTime}

We recall the classical result that the Laplace transform of exit times of intervals have a simple representation by eigenfunctions of $\cL$.
See also \cite[Chapter 5.15]{Ito_essentials}.

Fix $c > 0$ arbitrarily, and for $q \in \bC$ define $u = \Phi_{q}$ as the unique solution to the following integral equation:
\[
u(x) = 1 + q\int_{c}^{x}dy\int_{c}^{y}u(z)dm(z) \quad (x > 0).
\]
We can instead say that $u = \Phi_{q}$ is the unique solution to
\[
\cL u = q u , \quad u(c) = 1, \quad u^{+}(c) = 0.
\]
Similarly, define $u = \Psi_{q}$ as the unique solution to the following integral equation:
\[
u(x) = (x-c) + q\int_{c}^{x}dy\int_{c}^{y}u(z)dm(z) \quad (x > 0).
\]
The Wronskian of $\Phi_{q}$ and $\Psi_{q}$ equals one; $\Phi_{q}(x)\Psi_{q}^{+}(x) - \Phi_{q}^{+}(x)\Psi_{q}(x) = 1$.

We define the \textit{$q$-scale function} $W^{(q)}(x,y) \ (q \in \bC, \ x,y \in (0,\infty))$ by
\begin{align}	
W^{(q)}(x,y) := 
\begin{cases}
	\Phi_{q}(x)\Psi_{q}(y) - \Phi_{q}(y)\Psi_{q}(x) & (x \leq y) \\
	0 & (x > y)
\end{cases}
. \label{scaleFunc}
\end{align}

A notable feature of the $q$-scale function is that for each $x,y > 0$ the function $q \mapsto W^{(q)}(x,y)$  is entire.

\begin{Rem}
	The limit $\lim_{x \to 0}W^{(q)}(x,y)$ always exists by \cite[Theorem 5.13.4]{Ito_essentials}, and we denote it as $W^{(q)}(0,y)$.
	Note that the function $u = W^{(q)}(0,\cdot)$ satisfies $\cL u = qu,\ \lim_{y \to 0} u(y) = 0, \ \lim_{y \to 0}u^{+}(y) = 1$, and therefore $\psi_{q} = W^{(q)}(0,\cdot)$.
\end{Rem}

Note that for fixed $x > 0$, the function $u = W^{(q)}(x,\cdot)$ is the positive increasing solution to
\begin{align}	
\cL u(y) = q u(y) \ (y > x), \quad u(x) = 0, \quad u^{+}(x) = 1. \label{eq52}
\end{align}
Similarly, for fixed $y > 0$ the function $u = W^{(q)}(\cdot,y)$ is the positive decreasing solution to
\[
\cL u(x) = q u(x) \ (x < y), \quad u(y) = 0, \quad u^{+}(y) = -1.
\]
These imply $W^{(q)}(x,y) > 0$ for $0 < x < y$.
The above observations lead us to the following exit time formula (see \cite[Chapter 5]{Ito_essentials}).

\begin{Prop}[{\citet[Theorem 5.15.1]{Ito_essentials}}] \label{prop:two-sideExitTime}
	For $q \geq 0$ and $0 \leq x < y < z$, the following hold:
	\[
	\bE_{y}[\mathrm{e}^{-q T_{x}}, T_{x} < T_{z}] = \frac{W^{(q)}(y,z)}{W^{(q)}(x,z)} \quad \text{and} \quad  \bE_{y}[\mathrm{e}^{-q T_{z}}, T_{z} < T_{x}] = \frac{W^{(q)}(x,y)}{W^{(q)}(x,z)}.
	\]
\end{Prop}

The one-side exit time is also represented by an eigenfunction.

\begin{Prop}[{\citet[Theorem 5.15.1]{Ito_essentials}}]
	For $x > 0$ and $q \geq 0$, we have
\begin{align}
	g_{q}(x) = \bE_{x}[\mathrm{e}^{-q T_{0}}]. \label{eq16}
\end{align}
\end{Prop}

To analyze the behavior of the Laplace transform of the exit time with respect to $q$, the following simple formula on the scale function is quite useful.

\begin{Prop} \label{prop:resolventIdentity}
For $q,r \in \bC$ and $0 \leq x < y$, we have the following resolvent identity:
\begin{align}
	\begin{split}
	W^{(q)}(x,y) - W^{(r)}(x,y) &= (q-r)\int_{x}^{y}W^{(q)}(x,u)W^{(r)}(u,y)dm(u) \\
	&= (q-r)\int_{x}^{y}W^{(r)}(x,u)W^{(q)}(u,y)dm(u).
	\end{split}
	\label{eq55}
\end{align}
\end{Prop}

\begin{proof}
	We denote by $W^{(q)}_{i}(x,y) \ (i=1,2)$ for the right-derivative of $W^{(q)}(x,y)$ with respect to $i$-th variable.
	Note that
	\[
	W^{(q)}_{1}(x,y) = -1 - q \int_{x}^{y}W^{(q)}(u,y)dm(u) \quad \text{and} \quad
	W^{(q)}_{2}(x,y) = 1 + q \int_{x}^{y}W^{(q)}(x,u)dm(u).
	\]
	The case $q=0$ readily follows from \eqref{eq52}.
	Let $q \ne 0$.
	We have from the integration by parts formula
	\begin{align}
		&\int_{x}^{y}W^{(q)}(x,u)W^{(r)}(u,y)dm(u) \\
		 = &-\int_{x}^{y}W^{(r)}_{1}(u,y)du \int_{x}^{u}W^{(q)}(x,v)dm(v) \\
		= &-\int_{x}^{y}W^{(r)}_{1}(u,y) \frac{W^{(q)}_{2}(x,u)-1}{q} du \\
		= &-\frac{1}{q}\int_{x}^{y}W^{(q)}_{2}(x,u)W^{(r)}_{1}(u,y) du - \frac{W^{(r)}(x,y)}{q} \\
		= &-\frac{1}{q}\left( -W^{(q)}(x,y) - r\int_{x}^{y}W^{(q)}(x,u)W^{(r)}(u,y)dm(u) \right) - \frac{W^{(r)}(x,y)}{q}.
	\end{align}
	Rearranging the terms, we obtain the first equality in \eqref{eq55}.
	The second equality is given simply swapping $q$ and $r$ in the first equality.
\end{proof}

We show that the equality \eqref{eq16} is extended to $q \in (-\lambda_{0},0)$, which is important for our purpose.
To present the result, we first show the following

\begin{Thm} \label{thm:decayParameter}
	Define the decay parameter by
	\[
	\kappa := \sup \{ \lambda \geq 0 \mid \bE_{x}[\mathrm{e}^{\lambda T_{0}}] < \infty \ \text{for every (or some) $x > 0$} \}.
	\]
	Then we have
	\begin{align}
		\lambda_{0} = \kappa. \label{eq64}
	\end{align}
\end{Thm}

\begin{Rem}
	Theorem \ref{thm:decayParameter} is basically known: \cite[Theorem 4.1]{Littin} covers the entrance case, while \cite[Theorem 2]{Ratesofdecay} deals with the natural case.
	However, their arguments are quite different from ours.
\end{Rem}

\begin{Rem}
	The equality \eqref{eq64} generally fails, and $\kappa < \lambda_{0}$ holds when the state space is bounded (see \cite[Example 6.17]{Quasi-stationary_distributions}).
\end{Rem}

\begin{proof}[Proof of Theorem \ref{thm:decayParameter}]
	Set for $r > 0$
	\[
	\kappa_{r} := \sup \{ \lambda \geq 0 \mid \bE_{x}[\mathrm{e}^{\lambda T_{0}},T_{0} < T_{r}] < \infty \ (x \in (0,r) \}.
	\]

	Let $r >  0$ and $q \geq 0$.
	From Proposition \ref{prop:two-sideExitTime}, we have
	\begin{align}
		\bE_{x}[\mathrm{e}^{-qT_{0}}, T_{0} < T_{r}] = \frac{W^{(q)}(x,r)}{\psi_{q}(r)} \quad (x \in (0,r)). \label{eq53}
	\end{align}
	Let us focus on the RHS.
	As a function of $q$, the denominator has a zero at $q = -\lambda_{0}^{D,r}$.
	We see from the Strum-Liouville theory that for each $x \in (0,r)$
	\[
	\sup \{ \lambda \geq 0 \mid W^{(-\lambda)}(x,r) = 0 \} > \lambda_{0}^{D,r},
	\]
	and $W^{(-\lambda_{0}^{D,r})}(x,r) > 0$.
	Consequently, $W^{(q)}(x,r) / \psi_{q}(r)$ has a pole at $q = -\lambda_{0}^{D,r}$ and $\lim_{q \downarrow -\lambda_{0}^{D,r}}W^{(q)}(x,r) / \psi_{q}(r) = \infty$.

	We denote the power series expansion of $W^{(q)}(x,r) / \psi_{q}(r)$ around $0$ as $\sum_{n \geq 0}a_{n}q^{n}$.
	Let $\rho_{r}$ be the radius of convergence.
	We determine the coefficients $a_{n}$ through the LHS of \eqref{eq53}. 
	Clearly, $a_{0} = \bP_{x}[T_0 < T_r]$.
	For $a_{1}$, the monotone convergence theorem justifies the limit
	\[
	a_{1} = \lim_{q \downarrow 0} \bE_{x}\left[\frac{\mathrm{e}^{-qT_{0}}-1}{q},T_{0} < T_{r} \right]
	= -\bE_{x}[T_{0}, T_{0} < T_{r}],
	\]
	which in particular shows that $\bE_{x}[T_{0}, T_{0} < T_{r}]$ is finite.
	The coefficients $a_{n}$ ($n \geq 2$) are obtained inductively in the same way:
	once $a_{0},\dots,a_{n-1}$ and $\bE_{x}[T_0^{n-1},T_{0} < T_{r}] < \infty$ are known, we have
	\[
	a_{n} = \frac{(-1)^{n-1}}{n!}\lim_{q \downarrow 0}
	\bE_{x} \left[T_{0}^{n-1}\frac{e^{-qT_{0}} - 1}{q}, T_{0} < T_{r} \right]
	= \frac{(-1)^{n}}{n!}\bE_{x}[T_{0}^{n}, T_{0} < T_{r}].
	\]
	This shows that 
	\begin{align}
		\frac{W^{(-q)}(x,r)}{\psi_{-q}(r)} = \sum_{n \geq 0}\frac{q^{n}}{n!}\bE_{x}[T_{0}^{n}, T_{0} < T_{r}] \quad (|q| < \rho_{r}), \label{eq54}
	\end{align}
	and since the coefficients in the RHS of \eqref{eq54} are positive, the pole with the minimum absolute value is non-negative and coincides with the radius of convergence.
	That is, we obtain $\rho_{r} = \lambda_{0}^{D,r}$ and $\bE_{x}[\mathrm{e}^{\lambda_{0}^{D,r}T_{0}},T_{0} < T_{r}] = \infty$.
	This implies that $\lambda_{0}^{D,r} = \kappa_{r}$, and taking limit as $r \to \infty$, we have $\kappa \leq \lim_{r \to \infty}\kappa_{r} = \lim_{r \to \infty} \lambda_{0}^{D,r} = \lambda_{0}$.

	Suppose for the sake of contradiction that $\kappa < \lambda_{0}$.
	Note that this assumption implies $\kappa < \lambda_{0} < \lambda_{0}^{D,r} = \kappa_{r}$ for every $r > 0$.
	Let $\kappa < \lambda' < \lambda < \lambda_{0}$.
	We have from Proposition \ref{prop:two-sideExitTime} and \ref{prop:resolventIdentity} that
	\[
	1 - \frac{\psi_{-\lambda}(r)}{\psi_{-\lambda'}(r)} = (\lambda - \lambda')\int_{0}^{r}\psi_{-\lambda}(u)\bE_{u}[\mathrm{e}^{\lambda' T_{0}},T_{0} < T_{r}]dm(u).
	\]
	Since the LHS is less than $1$ by Proposition \ref{prop:monotonicityOfPsi},
	we have from the monotone convergence theorem
	\[
	\int_{0}^{\infty}\psi_{-\lambda}(u)\bE_{u}[\mathrm{e}^{\lambda' T_{0}}]dm(u) \leq \frac{1}{\lambda - \lambda'} < \infty,
	\]
	which implies $\kappa \geq \lambda'$.
	This is contradiction, and we obtain $\kappa = \lambda_{0}$.
\end{proof}

We present the announced result.

\begin{Prop}
	Suppose $\lambda_{0} > 0$.
	For $\lambda \in (0,\lambda_{0})$, we have
	\begin{align}
		\bE_{x}[\mathrm{e}^{\lambda T_{0}}] = \psi_{-\lambda}(x) \int_{x}^{\infty}\frac{dy}{\psi_{-\lambda}(y)^{2}} \quad (x > 0). \label{eq57}
	\end{align}
	In particular, 
	\begin{align}
		\int_{1}^{\infty}\frac{dy}{\psi_{-\lambda}(y)^{2}} < \infty. \label{eq18}
	\end{align}
\end{Prop}

\begin{proof}
	Let $q \geq 0$.
	Since $(0,r) \ni x \mapsto \bE_{x}[\mathrm{e}^{-qT_{0}}, T_{0} < T_{r}]$ is a positive decreasing solution to
	\[
	\cL u = qu, \quad \lim_{x \to 0}u(x) = 1, \quad \lim_{x \to r}u(x) = 0,
	\]
	we have the following representation: for every $x \in (0,r)$
	\begin{align}
	\bE_{x}[\mathrm{e}^{-qT_{0}}, T_{0} < T_{r}] = \psi_{q}(x)\int_{x}^{r}\frac{du}{\psi_{q}(u)^{2}}. \label{eq56}
	\end{align}
	The LHS is analytic in $q$ on $\{ z \in \bC \mid \Re z > -\lambda_{0} \}$.
	From Proposition \ref{prop:monotonicityOfPsi}, we can take a region $U \subset \bC$ such that $[-\lambda_{0},0] \subset U$ and
	\begin{align}
		\inf_{u \in [x,y], q \in \overline{U}} |\psi_{q}(u)| > 0, \label{}
	\end{align}
	where $\overline{U}$ denotes the closure of $U$, so that the function
	\begin{align}
		q \longmapsto \psi_{q}(x)\int_{x}^{r}\frac{du}{\psi_{q}(u)^{2}} \label{}
	\end{align}
	is analytic on $U$.
	Hence, the equality \eqref{eq56} is analytically extended to $q \in (-\lambda_{0},0)$.
	Fix $\lambda \in (0,\lambda_{0})$.
	Taking limit as $r \to \infty$ in \eqref{eq56} for $q = -\lambda$,
	we obtain the desired result from the monotone convergence theorem.
\end{proof}

\begin{Rem}
	By taking limit as $\lambda \to \lambda_{0}$, the equality \eqref{eq57} can be extended to $\lambda_{0}$.
	The limit is, however, not necessarily finite.
	When the boundary $\infty$ is entrance, the limit is infinite since $\psi_{-\lambda_{0}}$ is bounded.
	When the boundary $\infty$ is natural, both cases can occur where the limit is finite or infinite (see examples in Section \ref{section:examples}).
\end{Rem}

\section{Concave functions and $h$-transform} \label{section:concaveFunctions}

We introduce a class of concave functions central to our discussion and observe that such functions induce non-negative supermartingales.

\subsection{A class of concave functions} \label{section:classOfConcaveFunction}

We introduce a class of concave functions.

\begin{Def}\label{def:concaveFunctions}
Let $\cC$ denote the set of functions $\rho:(0,\infty) \to (0,\infty)$ that can be written as
\begin{align}
	\rho(x) 
	= \int_{0}^{x} dy \int_{y}^{\infty} \phi(z) dm(z)
	= \int_{0}^{\infty} (x \wedge y)\phi(y) dm(y)
	\quad (x > 0) \label{eq03}
\end{align}
for some non-negative, continuous, $m$-integrable function $\phi$ that is not identically zero.
\end{Def}

For clarity, we give another characterization of $\cC$.

\begin{Prop} \label{prop:characterizationOfC}
	A function $\rho: (0,\infty) \to (0,\infty)$ is in $\cC$ if and only if the following hold:
	\begin{enumerate}
		\item $\rho \in \cD(\cL)$ and $\phi := -\cL \rho \geq 0$.
		\item $\lim_{x \to 0}\rho(x) = 0$ and $\lim_{x \to 0}\rho^{+}(x) \in (0,\infty)$.
		\item $\lim_{x \to \infty}\rho^{+}(x) = 0$.
	\end{enumerate}
	In addition, every element $\rho \in \cC$ satisfies
	\begin{align}
	\lim_{x \to \infty}\rho(x)/x = 0. \label{eq61}
	\end{align}
\end{Prop}

\begin{proof}
	Since the direct implication is clear, we only prove the converse.
	From (i) and the definition of $\cL$, we have
	\begin{align}		
	\rho(x) = \rho(c) + \rho^{+}(c)(x-c) - \int_{c}^{x}dy\int_{c}^{y}\phi(z)dm(z) \quad (x > 0) \label{eq58}
	\end{align}
	and
	\[
	\rho^{+}(x) = \rho^{+}(c) - \int_{c}^{x}\phi(z)dm(z) \quad (x > 0)
	\]
	for some $c > 0$.
	The latter equation, with (ii), (iii) and the monotone convergence theorem, shows that 
	\[
	\rho^{+}(c) = \int_{c}^{\infty}\phi(z)dm(z) \quad \text{and} \quad \int_{0}^{c}\phi(z)dm(z) < \infty.
	\]
	In particular, $\phi$ is $m$-integrable.
	Plugging this into \eqref{eq58} and using (ii) and the monotone convergence theorem again, we have
	\[
	\rho(c) = \int_{0}^{c}dy\int_{y}^{\infty}\phi(z)dm(z) 
	\]
	Substituting this into \eqref{eq58} yields the representation \eqref{eq03}.
	The relation \eqref{eq61} follows by applying the dominated convergence theorem to \eqref{eq03}.
\end{proof}

\begin{Rem}
	The following potential formula is classical: for every non-negative function $f$ 
	\begin{align}		
	Uf(x) := \bE_{x}\left[ \int_{0}^{T_{0}}f(X_{s})ds \right] = \int_{0}^{\infty}(x \wedge y)f(y)dm(y) \quad (x > 0) \label{eq66}
	\end{align}
	(see e.g., \cite[Lemma 33.10]{Kallenberg-third}).
	Thus, we have $\cC = U(\cA)$ for 
	\begin{align}
	\cA := \{ \phi:(0,\infty) \to [0,\infty) \mid \phi \not\equiv 0, \ \phi \text{ is $m$-integrable, non-negative, and continuous} \}.			
	\end{align}
	We remark that the density function $x \wedge y$ in the RHS of \eqref{eq66} is obtained by substituting $q=0$ into \eqref{defOfResolventDensity}.
	Although this might suggest that class $\cC$ is obtained as the image of a certain set of non-negative $m$-integrable functions by the resolvent $R^{(q)} \ (q > 0)$, this is generally not the case.
	To see this, we consider the case of $m(1,\infty) = \infty$.
	Suppose for the sake of contradiction that $R^{(q)}f \in \cC$ for some $q > 0$ and non-negative $m$-integral function $f$.
	Then it follows that $R^{(q)}f$ is $m$-integrable since $R^{(q)}$ maps $L^{1}(m)$ to itself.
	Set  $\phi := -\cL(R^{(q)}f)$.
	We have by Fubini's theorem that
	\[
	\int_{0}^{\infty}R^{(q)}f(x)dm(x) = \int_{0}^{\infty}\phi(x)dm(x)\int_{0}^{x}m(y,\infty)dy = \infty,
	\]
	and this is contradiction.
	Consequently, we have $R^{(q)}(\cA') \cap \cC = \emptyset$ for any set $\cA'$ of non-negative $m$-integrable functions.
	Conversely, when $\infty$ is entrance, it follows from Proposition \ref{prop:eigenvaluesOfGenerator} that $0 \not\in \sigma(L_{1})$, and we can represent $\cA$ as the image of a certain set of $m$-integrable functions by $R^{(q)}$ for each $q > 0$.
\end{Rem}

From Proposition \ref{prop2-8}, we have the following

\begin{Prop}
Suppose $\lambda_{0}>0$ and let $q \in \bR$. 
Then $\psi_{q}\in\cC$ if and only if
\begin{enumerate}
    \item $\infty$ is entrance and $q=-\lambda_{0}$, or
    \item $\infty$ is natural and $q\in[-\lambda_{0},0)$.
\end{enumerate}
\end{Prop}

\subsection{Doob's $h$-transform by concave functions} \label{section:h-transform}

Since $\rho \in \cC$ is concave, It\^o formula shows that $\rho(X_{t})/\rho(x)$ is a non-negative supermartingale under $\bP_{x} \ (x > 0)$ (see e.g., \cite[Theorem 29.5]{Kallenberg-third}).
Thus, we can consider the change of measure by it.
However, what is more important here is that the resulting process again gives a one-dimensional diffusion.
Such a fact was investigated in \cite{Maeno} and \cite{TakemuraTomisaki:htransform} under the name of $h$-transform in a wide sense.

Let $\rho$ be an element of $\cC$.
Define the $h$-transform of $L$ by $\rho$ as
\begin{align}
	\tilde{L}^{[\rho]}u := \frac{1}{\rho}L(\rho u) \label{eq59}
\end{align}
for a function $u$ such that $\rho u \in \cD(L)$; $\cD(\tilde{L}^{[\rho]}) := \{ u:(0,\infty) \to \bR \mid \rho u \in \cD(L) \}$.
Then from \cite[Theorem 2.2]{Maeno}, the operator $\tilde{L}^{[\rho]}$ is again a generator of a one-dimensional diffusion.
Specifically, we have
\[
\tilde{L}^{[\rho]} = L^{(m^{[\rho]},s^{[\rho]},k^{[\rho]})} 
\]
for
\begin{align}
	dm^{[\rho]}(x) := \rho(x)^{2}dm(x), \quad ds^{[\rho]}(x) = \frac{dx}{\rho(x)^{2}}, \label{}
\end{align}
and
\[
dk^{[\rho]}(x) := \frac{\cL \rho(x)}{\rho(x)}dm^{[\rho]}(x),
\]
where the Dirichlet boundary condition is set at a boundary when it is regular for $L^{(m^{[\rho]},s^{[\rho]},k^{[\rho]})}$.
This implies 
\begin{align}
\tilde{\bP}^{[\rho]}_{x}[A] = \bE_{x}\left[
\frac{\rho(X_{t})}{\rho(X_{0})},A
\right] \quad (t \geq 0, \ A \in \cF_{t}), \label{eq45}
\end{align}
where $(\tilde{\bP}^{[\rho]}_{x})_{x > 0}$ is the underlying probability distributions of the $L^{(m^{[\rho]},s^{[\rho]},k^{[\rho]})}$-diffusion.
Define
\[
L^{[\rho]} := L^{(m^{[\rho]},s^{[\rho]},0)} =  \frac{d}{dm^{[\rho]}}\frac{d^{+}}{ds^{[\rho]}},
\]
that is, $L^{[\rho]}$ is given by removing the killing measure from $\tilde{L}^{[\rho]}$.

\begin{Prop} \label{prop:BoundaryClassification}
	Let $\rho \in \cC$. 
	The boundary $0$ is entrance for $L^{[\rho]}$.
	If the boundary $\infty$ is natural for $L$, then it is so for $L^{[\rho]}$.
\end{Prop}

\begin{proof}
	Since we have $cx \leq \rho(x) \leq Cx \ (0 \leq x \leq 1)$ for some $0 < c < C$ from Proposition \ref{prop:characterizationOfC} (ii), we easily see the boundary $0$ is entrance for $L^{[\rho]}$.
	Suppose $\infty$ is natural for $L$.
	Since $\rho$ is non-decreasing, we have
	\begin{align}
		\int_{1}^{\infty}ds^{[\rho]}(dx)\int_{x}^{\infty}dm^{[\rho]}(dy) = \int_{1}^{\infty}\frac{dx}{\rho(x)^{2}}\int_{x}^{\infty}\rho(y)^{2}dm(y) \geq \int_{1}^{\infty}dx\int_{x}^{\infty}dm(y) = \infty. \label{}
	\end{align}
	Since $\rho$ is concave and $\rho(0) = 0$, we have $\rho(x) / x$ is non-increasing.
	It follows
	\begin{align}
		\int_{1}^{\infty}dm^{[\rho]}(dx)\int_{x}^{\infty}ds^{[\rho]}(dy) &= \int_{1}^{\infty}\rho(x)^{2}dm(x)\int_{x}^{\infty}\frac{dy}{\rho(y)^{2}} \nonumber \\
		&\geq \int_{1}^{\infty}\rho(x)^{2}\frac{x^{2}}{\rho(x)^{2}}dm(x)\int_{x}^{\infty}\frac{dy}{y^{2}} \nonumber \\
		&= \int_{1}^{\infty}xdm(x) = \infty. \nonumber
	\end{align}
	Therefore, the boundary $\infty$ is natural for $L^{[\rho]}$.	
\end{proof}

We denote by $(\bP^{[\rho]}_{x})_{x > 0}$ the underlying probability distributions of diffusions corresponding to $L^{[\rho]}$.
Since the killing measure $dk^{[\rho]}$ induces killing by the continuous multiplicative functional 
\begin{align}	
N_{t}^{[\rho]} := \exp\left(\int_{0}^{t}\frac{\cL \rho(X_{u})}{\rho(X_{u})}du \right), \label{MFN}
\end{align}
the distribution of $L^{[\rho]}$-diffusion is given by a change of measure of $\tilde{\bP}_{x}^{[\rho]}$ by $N^{{\rho}}_{t}$
\[
\tilde{\bP}^{[\rho]}_{x}[A] = \bE^{[\rho]}[N_{t}^{[\rho]},A] \quad (t \geq 0, \ A \in \cF_{t})
\]
(see \cite[Chapter 5.6]{Ito-McKean}).
Combining this with \eqref{eq45}, we have
\begin{align}
	\bP_{x}^{[\rho]}[A] = \bE_{x}[M^{[\rho]}_{t},A] \quad (A \in \cF_{t}) \label{absoluteContinuityRelation}
\end{align}
for
\begin{align}
	M_{t}^{[\rho]} := \frac{\rho(X_{t})}{\rho(X_{0})} \exp \left(-\int_{0}^{t}\frac{\cL\rho(X_{s})}{\rho(X_{s})}ds \right), \label{martingale}
\end{align}
that is, the (sub)probability $\bP_{x}^{[\rho]}$ is a change of measure of $\bP_{x}$ by a supermartingale $M^{[\rho]}$.
Since $M_{t}^{[\rho]} > 0$ for $t < T_{0}$, the following holds for every bounded $\cF_{t}$-optional time $\tau$:
\begin{align}
	\bP_{x}[A,\tau < T_{0}] = \bE_{x}^{[\rho]}\left[ \frac{1_{A}}{M_{\tau}} \right] \quad (A \in \cF_{\tau}). \label{h-transformAbsContinuity}
\end{align}

Note that from Proposition \ref{prop:BoundaryClassification}, when the boundary $\infty$ is natural for $L$, the $L^{[\rho]}$-diffusion is conservative.
Also, when $\infty$ is entrance for $L$, the $L^{[\rho]}$-diffusion is conservative if and only if $\infty$ is entrance or natural for $L^{[\rho]}$.
In these cases, we have $\bE_{x} M^{[\rho]}_{t} = 1$ for every $t > 0$, and $M^{[\rho]}_{t}$ is a martingale under $\bP_{x}$.
We summarize this discussion as a proposition.

\begin{Prop} \label{prop:martingale}
	Let $\rho \in \cC$.
	The process $M^{[\rho]}_{t}$ is a non-negative $\cF_{t}$-supermartingale under $\bP_{x}$ for every $x > 0$.
	In addition, it is a $\cF_{t}$-martingale if and only if 
	\begin{enumerate}
		\item $\infty$ is entrance for $L$ and $\int_{1}^{\infty}\rho(x)^{2}dm(x)\int_{x}^{\infty}\frac{dy}{\rho(y)^{2}} = \infty$, or
		\item $\infty$ is natural for $L$.
	\end{enumerate}
\end{Prop}

\begin{Cor}
	The following hold:
	\begin{enumerate}
		\item When $\infty$ is entrance, $M_{t}^{[\psi_{-\lambda_{0}}]}$ is a $\cF_{t}$-martingale.
		\item When $\infty$ is natural and $\lambda_{0} > 0$, for every $\lambda \in (0,\lambda_{0}]$ $M_{t}^{[\psi_{-\lambda}]}$ is a $\cF_{t}$-martingale.
	\end{enumerate}
	Suppose $\lambda_{0} > 0$ and $\infty$ is entrance.
	Then $M_{t}^{[\psi_{-\lambda_{0}}]}$ is a $\cF_{t}$-martingale.
\end{Cor}

\begin{proof}
	Since $\psi_{-\lambda_{0}}$ is bounded from Proposition \ref{prop2-8}, we have $\int_{1}^{\infty}\frac{dy}{\psi_{-\lambda_{0}}(y)^{2}} = \infty$.
\end{proof}

\begin{Rem}
Assume $\infty$ is natural and $\lambda_{0} > 0$.
Let $\lambda \in (0,\lambda_{0})$.
Since $s^{[\psi_{-\lambda}]}(1,\infty) < \infty$ by \eqref{eq18} and $0$ is entrance for $L^{[\psi_{-\lambda}]}$, the process $X$ is transient to $\infty$ under $\bP^{[\psi_{-\lambda}]}_{x}$.
More precisely, we have from Proposition \ref{prop:BoundaryClassification} that 
\begin{align}
	\bP_{x}^{[\psi_{-\lambda}]}[X_{t} \in (0,\infty) \ \text{for every $t > 0$}]  = 1 \quad \text{and} \quad \bP_{x}^{[\psi_{-\lambda}]} \left[\lim_{t \to \infty} X_{t} =\infty \right] = 1. \label{eq35}
\end{align}
(see e.g., \citet[Theorem 33.15]{Kallenberg-third}).
Under $\bP_{x}^{[\psi_{-\lambda_{0}}]}$, both cases can occur where the process $X$ is transient or recurrent (see Section \ref{section:examples}).
\end{Rem}

From Proposition \ref{prop:BoundaryClassification}, when $\psi_{-\lambda} \in \cC$ for $\lambda \in (0,\lambda_{0}]$,
the boundary $0$ of $L^{[\psi_{-\lambda}]}$ is entrance.
This allows us to consider $\bP^{[\psi_{-\lambda}]}_{0}$, the distribution of $L^{[\psi_{-\lambda}]}$-diffusion starting from $0$ (see e.g., \citet[Theorem 33.13]{Kallenberg-third}).
For later use in Section \ref{section:examples}, we present a formula of the exit time under $\bP_{x}^{[\psi_{-\lambda}]} \ (x \geq 0)$.

\begin{Prop} \label{prop:HittingTimeUnderEntrance}
	Assume $\lambda_{0} > 0$.
	For $\lambda \in (0,\lambda_{0}]$, suppose $\psi_{-\lambda} \in \cC$.
	Then for $0 < x \leq y$, and $q \geq 0$, we have
	\begin{align}
		\bE_{x}^{[\psi_{-\lambda}]}[\mathrm{e}^{-q T_{y}}] = \frac{\psi_{q-\lambda}(x)}{\psi_{-\lambda}(x)} \frac{\psi_{-\lambda}(y)}{\psi_{q-\lambda}(y)} \quad \text{and} \quad \bE_{0}^{[\psi_{-\lambda}]}[\mathrm{e}^{-q T_{y}}] = \frac{\psi_{-\lambda}(y)}{\psi_{q-\lambda}(y)}.\label{}
	\end{align}
\end{Prop}

\begin{proof}
	It is classically known that
	\begin{align}
		\bE^{[\psi_{-\lambda}]}_{x}[\mathrm{e}^{-q T_{y}}] = \frac{\varphi^{(\lambda)}_{q}(x)}{\varphi^{(\lambda)}_{q}(y)}, \label{}
	\end{align}
	where $u = \varphi^{(\lambda)}_{q}$ is a unique solution of the following integral equation:
	\begin{align}
		u(x) = 1 + q \int_{0}^{x}ds^{[\psi_{-\lambda}]}(y)\int_{0}^{y}u(z)dm^{[\psi_{-\lambda}]}(z) \label{}
	\end{align}
	(see e.g., \citet[Chapter 5.15]{Ito_essentials}).
	By the definition of $dm^{[\psi_{-\lambda}]}$ and $ds^{[\psi_{-\lambda}]}$ we easily see that $\varphi^{(\lambda)}_{q} = \psi_{q-\lambda} / \psi_{-\lambda}$.
\end{proof}

\section{Conditional limits to avoid zero} \label{section:conditinalLimit}

We now get into our main problem.
The following is one of the main results in the present paper.
This theorem extends Theorem \ref{thm:conditionalLimitPrevious} and shows the existence of the conditioning to avoid zero by a clock induced by $\cC$.

\begin{Thm} \label{thm:conditionalLimit}
	Let $\rho \in \cC$.
	For $r > 0$, define an optional time
	\begin{align}
		S_{r}^{[\rho]} := \inf \left\{ t \geq 0 ~\middle|~  M_{t} \geq r / \rho(X_{0}) \right\}. \label{clock}
	\end{align}
	For every $s > 0$, $x > 0$ and $A \in \cF_{s}$, we have
	\begin{align}
		\lim_{r \to \infty} \bP_{x}[A \mid T_{0} > S_{r}^{[\rho]} ] = \bP_{x}^{[\rho]}[A]. \label{eq42}
	\end{align}
\end{Thm}

For the proof, we first show a key formula on the exit probability, which was given in \citet[Lemma 5]{Ratesofdecay} in the case $\rho = \psi_{-\lambda}$, and we generalize it.

\begin{Lem} \label{lem:exitTime}
Let $\rho \in \cC$ and let $r,x > 0$ such that $\rho(x) < r$.
For every $s > 0$, we have
\begin{align}
	\bP_{x}[S^{[\rho]}_{r} < T_{0} \wedge s ] = \frac{\rho(x)}{r}\bP_{x}^{[\rho]}[S_{r}^{[\rho]} < s]. \label{}
\end{align}
We also have
\begin{align}
	\bP_{x}^{[\rho]}[S_{r}^{[\rho]} < \infty] = 1 \quad \text{and} \quad \lim_{r \to \infty}\bP_{x}^{[\rho]}[S_{r}^{[\rho]} < s] = 0. \label{eq67}
\end{align}
Consequently, it follows
\begin{align}
	\lim_{r \to \infty}\frac{\bP_{x}[S^{[\rho]}_{r} < T_{0} \wedge s]}{\bP_{x}[S^{[\rho]}_{r} < T_{0}]} = 0. \label{eq60}
\end{align}
\end{Lem}

\begin{proof}
	Let $r,x > 0$ such that $\rho(x) < r$ and fix $s > 0$.
	It follows from \eqref{h-transformAbsContinuity} that
	\[
	\bP_{x}[S^{[\rho]}_{r} < T_{0} \wedge s ] = \bE_{x}^{[\rho]}\left[ \frac{1}{M_{S_{r}^{[\rho]}}},S_{r}^{[\rho]} < s \right] = \frac{\rho(x)}{r}\bP_{x}^{[\rho]}[ S_{r}^{[\rho]} < s].
	\]
	We show \eqref{eq67}.
	When $s^{[\rho]}(1,\infty) = \int_{1}^{\infty}dx / \rho(x)^{2} < \infty$, it follows $\lim_{x \to \infty}\rho(x) = \infty$.
	Since $\bP_{x}^{[\rho]}[\lim_{t \to \infty}X_{t} = \infty] = 1$, we have $\bP_{x}^{[\rho]}[\lim_{t \to \infty}M_{t}^{[\rho]} = \infty] = 1$, which implies $\bP_{x}^{[\rho]}[S_{r} < \infty] = 1$.
	When $s^{[\rho]}(1,\infty) = \infty$, the process $X$ is Harris recurrent under $\bP_{x}^{[\rho]}$: for every measurable set $A$ with $m^{[\rho]}(A) > 0$
	\[
	\int_{0}^{\infty}1_{A}(X_{t})dt = \infty \quad \bP_{x}^{[\rho]}\text{-a.s.}
	\]
	(see e.g., \cite[Theorem 26.17]{Kallenberg-third}).
	This implies
	\[
	-\int_{0}^{\infty}\frac{\cL \rho}{\rho}(X_{t})dt = \infty \quad \bP_{x}^{[\rho]}\text{-a.s.}
	\]
	since the function $-\cL\rho$ is non-negative, continuous, and not identically zero.
	Thus, we obtain $\bP_{x}^{[\rho]}[S_{r} < \infty] = 1$.
	Since  
	\[
	\sup_{t \in [0,s]}\rho(X_{t})\exp \left(-\int_{0}^{t}\frac{\cL\rho}{\rho}(X_{u})du\right) = \infty
	\]
	holds on $\lim_{r \to \infty} \{ S_{r}^{[\rho]} < s \}$, we readily see that $\lim_{r \to \infty}\bP_{x}^{[\rho]}[S_{r}^{[\rho]} < s] = 0$ from the continuity of $X_{t}$ and $\rho$.
\end{proof}

We prove Theorem \ref{thm:conditionalLimit}.

\begin{proof}[Proof of Theorem \ref{thm:conditionalLimit}]
	From \eqref{eq60}, it is enough to show
	\begin{align}
		\lim_{r \to \infty} r \bP_{x}[A, T_{0} > S^{[\rho]}_{r} > s ] = \rho(x)\bE_{x}\left[M^{[\rho]}_{s},A \right] \label{}
	\end{align}
	for $A \in \cF_{s}$.
	From the Markov property, we have $S^{[\rho]}_{r N_{s}^{[\rho]}} \circ \theta_{s} + s = S^{[\rho]}_{r}$ on $\{ S^{[\rho]}_{r} > s \}$,
	where $N^{[\rho]}_{s}$ is a multiplicative functional given in \eqref{MFN}.
	Lemma \ref{lem:exitTime} and the monotone convergence theorem give
	\begin{align}
		r\bP_{x}[A , T_{0} > S^{[\rho]}_{r} > s ] = &r\bE_{x}[\bP_{X_{s}}[T_{0} > S^{[\rho]}_{r u}]_{u = N_{s}^{[\rho]}}1_{A}, S^{[\rho]}_{r} \wedge T_{0} > s] \label{} \\
		= &\bE_{x}\left[ (N_{s}^{[\rho]})^{-1}\rho(X_{s})1_{A}, S^{[\rho]}_{r} \wedge T_{0} > s \right] \label{} \\
		\xrightarrow{r \to \infty} &\rho(x)\bE_{x}\left[ M^{[\rho]}_{s},A \right], \label{}
	\end{align}
	and we obtain the desired result.
\end{proof}

\section{Absolute continuity relations of limit distributions} \label{section:absoluteContinuity}

The map $\cC \ni \rho \mapsto (\bP_{x}^{[\rho]})_{x > 0}$ is injective modulo multiplication of positive constants to $\rho$.
Here we consider a classification by absolute continuity relation on $\cF_{\infty}$, that is, we investigate when two distributions $\bP_{x}^{[\alpha]}$ and $\bP_{x}^{[\beta]}$ for $\alpha,\beta \in \cC$ are absolutely continuous or singular on $\cF_{\infty}$.

We first consider the case where $\alpha \in \cC$ satisfies $s^{[\alpha]}(1,\infty) < \infty$.
In this case, perturbations preserve the absolute continuity.

\begin{Thm} \label{thm:mutualAbsConti}
	Let $\alpha, \beta \in \cC$.
	Suppose $s^{[\alpha]}(1,\infty) < \infty$ and the boundary $\infty$ is natural for $L^{[\alpha]}$ and $L^{[\beta]}$.
	Set
	\begin{align}
		\delta(x) := \frac{\cL\alpha(x)}{\alpha(x)} - \frac{\cL\beta(x)}{\beta(x)}, \label{eq34}
	\end{align}
	and assume
	\begin{align}
		&\int_{1}^{\infty}(\alpha\beta|\delta|)(x)dm(x)\int_{x}^{\infty}\frac{dy}{\alpha(y)^{2}} < \infty \label{eq29}
	\end{align}
	and
	\begin{align}
		\int_{0}^{\infty}(\alpha^{2}\delta_{-})(x)dm(x)\int_{x}^{\infty}\frac{dy}{\alpha(y)^{2}} < 1, \label{eq33}
	\end{align}
	where $\delta_{-} := -(\delta \wedge 0)$.
	Then limit $c := \lim_{x \to \infty} \beta(x)/\alpha(x) \in (0,\infty
	)$ exists and the following hold for every $x > 0$:
	\begin{align}
		 \bE_{x}^{[\alpha]}\left[ \exp \left(\int_{0}^{\infty} \delta_{-}(X_{s})ds\right) \right] < \infty \nonumber
	\end{align}
	and 
	\begin{align}
		\bP_{x}^{[\beta]}[A] = c \frac{\alpha(x)}{\beta(x)}\bE^{[\alpha]}_{x} \left[ \exp \left( -\int_{0}^{\infty} \delta(X_{s})ds \right), A\right] \quad (A \in \cF_{\infty}). \label{eq31}
	\end{align}
	Consequently, the distribution $\bP_{x}^{[\beta]}$ is absolutely continuous with respect to $\bP_{x}^{[\alpha]}$ on $\cF_{\infty}$.
\end{Thm}

\begin{Rem}
	Under the assumptions of Theorem \ref{thm:mutualAbsConti}, we have $s^{[\beta]}(1,\infty) < \infty$ since $\lim_{x \to \infty}\beta(x)/\alpha(x) \in (0,\infty)$.
\end{Rem}

\begin{proof}[Proof of Theorem \ref{thm:mutualAbsConti}]
	We first show limit $\lim_{n \to \infty} \beta(x)/\alpha(x)$ exists.
	We have	
	\begin{align}
		\left| \left( \frac{\beta}{\alpha} \right)^{+} (x) \right|
		= \left| \frac{\alpha \beta^{+} - \alpha^{+}\beta}{\alpha^{2}} (x)\right| = \frac{\left| \int_{0}^{x}(\alpha (\cL\beta) - (\cL\alpha)\beta )dm \right|}{\alpha(x)^{2}} \leq \frac{\int_{0}^{x}(\alpha\beta|\delta|) dm}{\alpha(x)^{2}}. \label{}
	\end{align}
	Since the rightmost term is integrable w.r.t.\ the Lebesgue measure on $[1,\infty)$ from \eqref{eq29}, the following limit exists from the dominated convergence theorem:
	\begin{align}
		c := \lim_{x \to \infty}\frac{\beta}{\alpha}(x) = \frac{\beta}{\alpha}(1) + \int_{1}^{\infty}\left( \frac{\beta}{\alpha} \right)^{+}(y)dy \in [0,\infty). \label{}
	\end{align}
	By the definition of $\bP^{[\alpha]}_{x}$ and $\bP^{[\beta]}_{x}$,
	we have for $0 < t \leq u$ and $A \in \cF_{t}$
	\begin{align}
		\bP_{x}^{[\beta]}[A] = \frac{\alpha(x)}{\beta(x)}\bE^{[\alpha]}_{x} \left[ \frac{\beta(X_{u})}{\alpha(X_{u})} \exp \left( -\int_{0}^{u} \delta_{+}(X_{s})ds + \int_{0}^{u} \delta_{-}(X_{s}) ds \right), A \right], \label{eq27c}
	\end{align}
	where $\delta_{+} := \delta \vee 0$.
	Since the process $X$ is transient to $\infty$ under $\bP^{[\alpha]}_{x}$ by the assumption that $s^{[\alpha]}(1,\infty) < \infty$ and $\beta / \alpha$ is bounded,
	if we can show
	\begin{align}
		\bE_{x}^{[\alpha]} \left[ \exp \left( \int_{0}^{\infty} \delta_{-}(X_{s})ds \right) \right] < \infty, \label{eq32}
	\end{align}
	we may apply the dominated convergence theorem to \eqref{eq27c} for limit $u \to \infty$ and obtain \eqref{eq31} for $A \in \cF_{t}$.
	That implies $c > 0$ and the desired result since $t > 0$ and $A \in \cF_{t}$ are arbitrary.
	Thus, we show \eqref{eq32}.
	By Kac's moment formula (see \citet[(4)]{FitzsimmonsPitman-KacMomentFormula}),
	we have
	\begin{align}
		\mu_{n} := \bE_{x}^{[\alpha]} \left[ \left(\int_{0}^{\infty}\delta_{-}(X_{s})ds \right)^{n} \right] = n! (G^{n}_{\delta_{-}}1)(x), \label{}
	\end{align}
	where $1(x) \equiv 1$ is a constant function and $G_{\delta_{-}}$ is an  operator defined by
	\begin{align}
		G_{\delta_{-}}f(x) := \bE_{x}^{[\alpha]} \left[ \int_{0}^{\infty}f(X_{s})\delta_{-}(X_{s})ds \right] \label{}
	\end{align} 
	for non-negative measurable function $f$.
	Thus, it suffices to show $\mu_{n} / n! \leq \eps^{n}$ for some $\eps \in (0,1)$, which is known to be Khas'minskii's condition (see \citet[p.120]{FitzsimmonsPitman-KacMomentFormula}).
	By a potential formula (see e.g., \citet[Lemma 33.10]{Kallenberg-third}), the operator $G_{\delta_{-}}$ is represented as an integral operator:
	\begin{align}
		G_{\delta_{-}}f(x) &= \int_{0}^{\infty} f(y)\delta_{-}(y)s^{[\alpha]}(x \vee y,\infty)dm^{[\alpha]}(y) \label{} \label{} \\
		&= \int_{x}^{\infty}ds^{[\alpha]}\int_{0}^{y}f(z)\delta_{-}(z)dm^{[\alpha]}(z), \label{}
	\end{align}
	where we used Fubini's theorem in the second equality.
	Then we have from the assumption \eqref{eq33} that
	\begin{align}
		G_{\delta_{-}}1(x) \leq G_{\delta_{-}}1(0) < 1. \label{}
	\end{align}
	We inductively see that
	\begin{align}
		G^{n}_{\delta_{-}}1(x) &= \int_{x}^{\infty}ds^{[\alpha]}\int_{0}^{y}(G^{n-1}_{\delta_{-}}1)(z)\delta_{-}(z)dm^{[\alpha]}(z) \label{} \\
		&\leq (G_{\delta_{-}}1(0))^{n-1} \int_{x}^{\infty}ds^{[\alpha]}\int_{0}^{y}\delta_{-}(z)dm^{[\alpha]}(z) \label{} \\
		&\leq (G_{\delta_{-}}1(0))^{n}. \label{}
	\end{align}
	The proof is complete.
\end{proof}

Roughly speaking, the assumptions in Theorem \ref{thm:mutualAbsConti} are satisfied when the difference $|\cL\alpha/\alpha(x)- \cL\beta/\beta(x)|$ vanishes sufficiently fast as $x$ goes to infinity.
When $\alpha = \psi_{-\lambda}$, and $\beta = \psi_{-\mu}$ for $\lambda \neq \mu \ (\lambda,\mu \in (0,\lambda_{0}])$, this is not the case; $|\cL\psi_{-\lambda}/\psi_{-\lambda}(x)- \cL\psi_{-\mu}/\psi_{-\mu}(x)| \equiv |\lambda - \mu| > 0$.
Conversely, in this case, we show that the distributions $\bP_{x}^{[\psi_{-\lambda}]}$ and $\bP_{x}^{[\psi_{-\mu}]}$ are singular on $\cF_{\infty}$ under the \textit{tail triviality}.
Let $\cT := \bigcap_{t > 0} \sigma \left(\bigcup_{s > t} \cF_{s}\right)$, the tail $\sigma$-field of $\cF_{t}$.
For $\rho \in \cC$, the $\sigma$-field $\cT$ is said to be trivial under $\bP^{[\rho]}_{x}$ when 
\begin{align}
	\bP_{x}^{[\rho]}[A] \in  \{0,1\} \quad \text{for every $A \in \cT$}. \label{tailTrivial}
\end{align}

Several necessary and sufficient conditions for the tail triviality of one-dimensional diffusions have been established in \cite{Rosler-tailTrivial}, \cite{FristedtOrey-tailTrivial} and \cite{Rogers-tailTrivial}.
Specializing their results in our situation, we have the following.

\begin{Thm}[{\citet[Theorem 2.2]{Rosler-tailTrivial}, \citet[Theorem 3]{Rogers-tailTrivial}}] \label{thm:tailTrivialNeceSuff}
	For $\rho \in \cC$, the tail triviality \eqref{tailTrivial} holds if and only if either of the following holds:
	\begin{enumerate}
		\item $s^{[\rho]}(1,\infty) = \infty$.
		\item $s^{[\rho]}(1,\infty) < \infty$ and 
		\begin{align}
			\int_{1}^{\infty}s^{[\rho]}(x,\infty)m^{[\rho]}(0,x)^2 ds^{[\rho]}(x) = \infty. \label{eq23}
		\end{align}
	\end{enumerate}
\end{Thm}

We present a sufficient condition for the tail triviality in terms of the speed measure $dm$ and the scale function $s(x) = x$ of the original process.

\begin{Prop} \label{prop:tailTrivialSuff}
	Let $\rho \in \cC$ and suppose $s^{[\rho]}(1,\infty) < \infty$.
	If 
	\begin{align}
		\int_{1}^{\infty}\left( \int_{1}^{x}y^{2}dm(y) \right)^{2}\frac{dx}{x^{3}} = \infty, \label{}
	\end{align}
	then \eqref{eq23} holds.
	In particular, 
	\begin{align}
		\liminf_{x \to \infty} \frac{1}{x}\int_{1}^{x}y^{2}dm(y) > 0 \label{}
	\end{align}
	is a sufficient condition for \eqref{eq23}.
\end{Prop}

\begin{proof}
	By the concavity of $\rho$, we have $\rho(z)/z \leq \rho(y) / y$ for $z \geq y$.
	It follows
	\begin{align}
		s^{[\rho]}(x,\infty)m^{[\rho]}(1,x) 
		= \int_{x}^{\infty} \frac{dz}{\rho(z)^{2}} \int_{1}^{x}\rho(y)^{2}dm(y) \geq \frac{1}{x}\int_{1}^{x} y^{2}dm(y).  \label{}
	\end{align}
	Similarly, we see that
	\begin{align}
		\frac{1}{\rho(x)^{2}}m^{[\rho]}(1,x) \geq \frac{1}{x^{2}}\int_{1}^{x}y^{2}dm(y). \label{}
	\end{align}
	Now the desired result is clear.
\end{proof}

The following theorem shows that mutual singularity on $\cF_{\infty}$ follows under the tail triviality.
Although its proof is straightforward, the importance lies in the fact that without this assumption, singularity may fail, while mutual absolute continuity can hold instead, as will be seen in the examples in Section \ref{section:examples}.

\begin{Thm} \label{thm:singular}
	Assume $\infty$ is natural for $L$ and $\lambda_{0} > 0$.
	For $\mu, \lambda \in (0,\lambda_{0}]$ with $\lambda \neq \mu$ and $x > 0$,
	suppose the tail $\sigma$-field $\cT$ is trivial under $\bP_{x}^{[\psi_{-\lambda}]}$ and $\bP_{x}^{[\psi_{-\mu}]}$. 
	Then the distributions $\bP_{x}^{[\psi_{-\lambda}]}$ and $\bP_{x}^{[\psi_{-\mu}]}$ are singular on $\cF_{\infty}$.
	More explicitly, there exist some constants $0 \leq c_{1} < 1 < c_{2} \leq \infty$ such that
	\begin{align}
		\bP_{x}^{[\psi_{-\lambda}]}\left[ \lim_{t \to \infty} \frac{M^{[\psi_{-\mu}]}_{t}}{M^{[\psi_{-\lambda}]}_{t}} = c_{1} \right] =
		\bP_{x}^{[\psi_{-\mu}]}\left[ \lim_{t \to \infty} \frac{M^{[\psi_{-\mu}]}_{t}}{M^{[\psi_{-\lambda}]}_{t}} = c_{2} \right] = 1. \label{}
	\end{align}
\end{Thm}

\begin{proof}
	Since $M^{[\psi_{-\mu}]}_{t} / M^{[\psi_{-\lambda}]}_{t}$ is a non-negative martingale under $\bP_{x}^{[\psi_{-\lambda}]}$, it converges almost surely as $t \to \infty$.
	Since we have
	\begin{align}
		\frac{M^{[\psi_{-\mu}]}_{t}}{M^{[\psi_{-\lambda}]}_{t}} = \mathrm{e}^{(\mu-\lambda)t} \frac{\psi_{-\lambda}(x)}{\psi_{-\mu}(x)} \frac{\psi_{-\mu}(X_{t})}{\psi_{-\lambda}(X_{t})} \quad \bP_{x}^{[\psi_{-\lambda}]}\text{-a.s.}, \label{}
	\end{align}
	limit $\lim_{t \to \infty}M^{[\psi_{-\mu}]}_{t} / M^{[\psi_{-\lambda}]}_{t}$ is $\cT$-measurable.
	From the triviality of $\cT$ under $\bP_{x}^{[\psi_{-\lambda}]}$, the limit is almost surely a constant, which we denote by $c_{1}$.
	By Fatou's lemma $c_{1} \leq M^{[\psi_{-\mu}]}_{0} / M^{[\psi_{-\lambda}]}_{0} = 1$.
	If $c_{1} = 1$, the martingale $M^{[\psi_{-\mu}]}_{t} / M^{[\psi_{-\lambda}]}_{t}$ is uniformly integrable, and it implies $M^{[\psi_{-\mu}]}_{t} / M^{[\psi_{-\lambda}]}_{t} = 1, \ \bP_{x}^{[\psi_{-\lambda}]}$-a.s. for every $t>0$.
	This is obviously impossible and thus $c_{1} < 1$.
	The same consideration for the martingale $M^{[\psi_{-\lambda}]}_{t} / M^{[\psi_{-\mu}]}_{t}$ under $\bP_{x}^{[\psi_{-\mu}]}$ shows 
	\begin{align}
		\bP_{x}^{[\psi_{-\mu}]}\left[ \lim_{t \to \infty} \frac{M^{[\psi_{-\lambda}]}_{t}}{M^{[\psi_{-\mu}]}_{t}} = \frac{1}{c_{2}} \right] = 1 \label{}
	\end{align}
	for some $c_{2} \in (1,\infty]$.
\end{proof}

We consider the case where $\alpha \in \cC$ satisfies $s^{[\alpha]}(1,\infty) = \infty$.
We always have singularity on $\cF_{\infty}$ in this situation.

\begin{Thm}
	Let $\alpha, \beta \in \cC$ with $\alpha \ne \beta$.
	Suppose $s^{[\alpha]}(1,\infty) = \infty$.
	Then $\bP_{x}^{[\alpha]}$ and $\bP_{x}^{[\beta]}$ are singular on $\cF_{\infty}$ for every $x > 0$.
\end{Thm}

\begin{proof}
	Note that from $s^{[\alpha]}(1,\infty) = \infty$, the $L^{[\alpha]}$-diffusion is recurrent.

	If $s^{[\beta]}(1,\infty) < \infty$, the $L^{[\beta]}$-diffusion is transient to $\infty$, and we immediately have singularity on $\cF_{\infty}$ from
	\[
	1 - \bP_{x}^{[\alpha]}[T_{\infty} < \infty] = \bP_{x}^{[\beta]}[T_{\infty} < \infty] = 1,
	\]
	where $T_{\infty} := \lim_{x \to \infty}T_{x}$.

	Let $s^{[\beta]}(1,\infty) = \infty$.
	We may assume without loss of generality 
	\begin{align}
		\lim_{x \to 0}\alpha^{+}(x) = \lim_{x \to 0}\beta^{+}(x). \label{eq63}
	\end{align}
	Since $\alpha$ and $\beta$ are continuous, there exists a finite interval $I$ and $\eps > 0$ such that $|\alpha(x) - \beta(x)| > \eps \ (x \in I)$.
	From \eqref{eq63}, we can take a sufficiently small $\delta > 0$ so that
	\[
		\frac{\int_{I}\alpha(x)^{2}dm(x)}{\int_{0}^{\delta}\alpha(x)^{2}dm(x)} \ne \frac{\int_{I}\beta(x)^{2}dm(x)}{\int_{0}^{\delta}\beta(x)^{2}dm(x)}.
	\]
	From the ratio ergodic theorem for recurrent diffusions (see e.g., \cite[Chapter 6.8]{Ito-McKean} and \cite[Theorem 33.14]{Kallenberg-third}),
	we have for $\gamma = \alpha,\beta$
	\[
	\lim_{t \to \infty}\frac{\int_{0}^{t}1_{I}(X_{s})ds}{\int_{0}^{t}1_{(0,\delta]}(X_{s})ds} = \frac{\int_{I}\gamma(x)^{2}dm(x)}{\int_{0}^{\delta}\gamma(x)^{2}dm(x)} \quad \bP_{x}^{[\gamma]}\text{-a.s.},
	\]
	which shows the desired singularity.
\end{proof}

\section{Examples} \label{section:examples}

We consider the results discussed above in two examples.

\subsection{Brownian motion with a negative drift}

For $c > 0 $, let us consider a Brownian motion with a constant negative drift killed at $0$, that is,
\[
X_{t} := B_{t} - ct \quad (t < T_{0}).
\]
Its local generator on $(0,\infty)$ is
\begin{align}
	\cL = \frac{1}{2}\frac{d^{2}}{dx^{2}} - c \frac{d}{dx}, \label{eq40}
\end{align}
and its speed measure and scale function are given by
\begin{align}
	dm(x) = 2 \mathrm{e}^{-2cx} dx, \quad ds(x) = \mathrm{e}^{2cx}dx , \label{} 
\end{align}
respectively.
The boundary $0$ is regular and the boundary $\infty$ is natural.
The bottom of the spectrum $\lambda_{0}$ of $L^{2}(m)$ is $c^{2}/2$.
An elementary calculus shows that the eigenfunction $\psi_{q}$ in \eqref{eq12} is
\begin{align}
	\psi_{q}(x) =
	\begin{cases}
		\frac{\mathrm{e}^{cx}}{\sqrt{c^{2} + 2q}} \sinh (\sqrt{c^{2}+ 2q}x) & (q \neq -c^{2}/2) \\
		x \mathrm{e}^{cx} & (q = -c^{2}/ 2)
	\end{cases} 
	. \label{eq39}
\end{align}
We thus have for $\lambda \in (0,c^{2}/2)$
\begin{align}
	dm^{[\psi_{-\lambda}]}(x) = \frac{2 \sinh (\sqrt{c^{2}- 2\lambda}x)^{2}}{c^{2}-2\lambda}dx , \quad ds^{[\psi_{-\lambda}]}(x) = \frac{c^{2}-2\lambda}{\sinh (\sqrt{c^{2}- 2\lambda}x)^{2}}dx \label{}
\end{align}
and 
\[
dm^{[\psi_{-c^{2}/2}]}(x) = 2x^{2} dx , \quad ds^{[\psi_{-c^{2}/2}]}(x) = \frac{dx}{x^{2}}.
\]
It follows
\begin{align}
	\frac{d}{dm^{[\psi_{-\lambda}]}}\frac{d}{ds^{[\psi_{-\lambda}]}} = 
	\begin{cases}
		\frac{1}{2}\frac{d^{2}}{dx^{2}} + \sqrt{c^{2}-2\lambda} \coth (\sqrt{c^{2}-2\lambda}x)\frac{d}{dx} & (\lambda \in (0,c^{2}/2)) \\
		\frac{1}{2}\frac{d^{2}}{dx^{2}} + \frac{1}{x}\frac{d}{dx} & (\lambda = c^{2}/2)
	\end{cases}
	. \label{}
\end{align}
Thus, the process $X$ is a hyperbolic Bessel process under $\bP_{x}^{[\psi_{-\lambda}]} \ (\lambda \in (0,c^{2}/2))$ and Bessel process of dimension three under $\bP_{x}^{[\psi_{-c^{2}/2}]}$ (see e.g., \citet[Chapter VIII.3]{RevuzYor}).
Note that $X$ is transient under $\bP_{x}^{[\psi_{-c^{2}/2}]}$.
The process $X$ satisfies the tail triviality under $\bP_{x}^{[\rho]}$ for every $\rho \in \cC$ since
\begin{align}
	\frac{1}{s(0,x)}\int_{1}^{x}s(0,y)^{2}dm(y) = \frac{1}{c(\mathrm{e}^{2cx}-1)} \int_{1}^{x}(\mathrm{e}^{2cy}-1)^{2}\mathrm{e}^{-2cy}dy \xrightarrow[]{x \to \infty} \frac{1}{2c^{2}}, \label{}
\end{align}
and we can apply Proposition \ref{prop:tailTrivialSuff}.

We show that the constant $c_{1}$ in Theorem \ref{thm:singular} is always zero for this process.

\begin{Thm} \label{thm:TailConstantBM}
	For $\lambda, \mu \in (0,c^{2}/2]$ with $\lambda \neq \mu$ and $x \geq 0$, we have
	\begin{align}
		\bP^{[\psi_{-\lambda}]}_{x}\left[ \lim_{t \to \infty} \frac{M_{t}^{[\psi_{-\mu}]}}{M_{t}^{[\psi_{-\lambda}]}} = 0 \right] = 1. \label{}
	\end{align}
\end{Thm}

We use the following lemma.

\begin{Lem} \label{lem:AsympOfHittingTimeBM}
	For $\lambda \in (0,c^{2}/2)$, we have
	\begin{align}
		\frac{T_{y}}{y} \xrightarrow[y \to \infty]{P} \frac{1}{\sqrt{c^{2}-2\lambda}} \quad \text{under } \bP_{0}^{[\psi_{-\lambda}]}. \label{}
	\end{align}
	Under $\bP_{0}^{[\psi_{-c^{2}/2}]}$, the distribution of $T_{y}/y^{2}$ does not depend on $y$ and satisfies
	\begin{align}
		\bE_{0}^{[\psi_{-c^{2}/2}]}[\mathrm{e}^{-q T_{y}/y^{2}}] = \frac{\sqrt{2\beta}}{\sinh \sqrt{2q}} \quad (q \geq 0). \label{eq41}
	\end{align}
\end{Lem}

\begin{proof}
	Let $\lambda,\mu < c^{2}/2$.
	Then it follows from \eqref{eq39} and Proposition \ref{prop:HittingTimeUnderEntrance}
	\begin{align}
		\bE_{0}^{[\psi_{-\lambda}]}[\mathrm{e}^{-q(T_{y}/y)}]
		&= \frac{\psi_{-\lambda}(y)}{\psi_{q/y-\lambda}(y)} \label{} \\
		&= \sqrt{\frac{c^{2}-2(\lambda-q/y)}{c^{2}-2\lambda}}\frac{\sinh (\sqrt{c^{2} -2\lambda}y)}{\sinh (\sqrt{c^{2}-2(\lambda - q/y)}y)} \label{} \\
		&\sim \exp \left( (\sqrt{c^{2} -2\lambda} - \sqrt{c^{2}-2(\lambda - q/y)})y \right) \label{} \\
		&\sim \mathrm{e}^{-q/\sqrt{c^{2}-2\lambda}} \quad (y \to \infty). \label{}
	\end{align}
	Thus, $T_{y}/y \xrightarrow[y \to \infty]{P} (c^{2}-2\lambda)^{-1/2}$ under $\bP_{0}^{[\psi_{-\lambda}]}$.
	The equality \eqref{eq41} also easily follows from \eqref{eq39} and Proposition \ref{prop:HittingTimeUnderEntrance}.
\end{proof}

We prove Theorem \ref{thm:TailConstantBM}.

\begin{proof} [Proof of Theorem \ref{thm:TailConstantBM}]
	It is enough to show the case $x = 0$.
	The general case follows from the strong Markov property at $T_{x}$.
	Noting that $\lim_{x \to 0} \psi_{-\mu}(x)/\psi_{-\lambda}(x) = 1$,
	we see that for every $t > 0$
	\begin{align}
		\bE_{0}^{[\psi_{-\lambda}]}\left[ \mathrm{e}^{(\mu-\lambda)t} \frac{\psi_{-\mu}(X_{t})}{\psi_{-\lambda}(X_{t})},A \right] = \bP_{0}^{[\psi_{-\mu}]} A \quad (A \in \cF_{t}). \label{}
	\end{align}
	Since the $\bP_{0}^{[\psi_{-\lambda}]}$-martingale $W_{t} := \mathrm{e}^{(\mu-\lambda)t} \psi_{-\mu}(X_{t}) / \psi_{-\lambda}(X_{t})$ converges almost surely as $t \to \infty$, we denote the limit random variable by $Z$.
	From \eqref{eq35}, we have $\bP_{0}^{[\psi_{-\lambda}]}[\lim_{y \to \infty} W_{T_{y}} = Z] = 1$.
	Thus, it suffices to identify the limit distribution of $W_{T_{y}}$ as $y \to \infty$.
	
	Let $\lambda, \mu < c^{2}/2$.	
	From \eqref{eq39}, we have as $y \to \infty$
	\begin{align}
		W_{T_{y}} &= \mathrm{e}^{(\mu - \lambda)T_{y}}\frac{\psi_{-\mu}(y)}{\psi_{-\lambda}(y)} \label{} \\
		&\sim \sqrt{\frac{c^{2}-2\lambda}{c^{2}-2\mu}} \exp\left((\mu - \lambda)T_{y} + (\sqrt{c^{2}-2\mu} - \sqrt{c^{2}-2\lambda})y  \right) \label{} \\
		&\sim \sqrt{\frac{c^{2}-2\lambda}{c^{2}-2\mu}} \exp\left((\mu - \lambda)\left(\frac{T_{y}}{y} - \frac{2}{\sqrt{c^{2}-2\mu} + \sqrt{c^{2}-2\lambda}}\right)y \right). \label{}
	\end{align}
	From Lemma \ref{lem:AsympOfHittingTimeBM}, we have under $\bP_{0}^{[\psi_{-\lambda}]}$ that
	\begin{align}
		(\mu - \lambda)\left(\frac{T_{y}}{y} - \frac{2}{\sqrt{c^{2}-2\mu} + \sqrt{c^{2}-2\lambda}}\right) \xrightarrow[y \to \infty]{P} -\infty. \label{}
	\end{align}
	Thus, we obtain $\bP_{0}^{[\psi_{-\lambda}]}[Z=0] = 1$.
	The case $\lambda < c^{2}/2$ and $\mu = c^{2}/2$ is almost the same, and we omit it.
	Consider the case $\lambda = c^{2}/2$ and $\mu < c^{2}/2$.
	We have as $y \to \infty$
	\begin{align}
		W_{T_{y}} &= \mathrm{e}^{(\mu - c^{2}/2)T_{y}}\frac{\psi_{-\mu}(y)}{\psi_{-c^{2}/2}(y)} \label{} \\
		&= \mathrm{e}^{(\mu-c^{2}/2)T_{y}} \frac{\sinh (\sqrt{c^{2}-2\mu}y)}{y\sqrt{c^{2}-2\mu}} \label{} \\
		&\sim \frac{1}{2y\sqrt{c^{2}-2\mu}}\exp\left(\left(\mu-\frac{c^{2}}{2}\right)T_{y} -\sqrt{c^{2}-2\mu}y\right) \label{} \\
		&= \frac{\mathrm{e}^{y^{2}/(2T_{y})}}{2y\sqrt{c^{2}-2\mu}}\exp\left(-\frac{T_{y}}{2}\left(\sqrt{c^{2}-2\mu}-\frac{y}{T_{y}}\right)^{2} \right). \label{}
	\end{align}
	A little thought with Lemma \ref{lem:AsympOfHittingTimeBM} shows $W_{T_{y}} \xrightarrow[y \to \infty]{P} 0$ under $\bP_{0}^{{[\psi_{-c^{2}/2}]}}$.
\end{proof}

\subsection{Ornstein-Uhlenbeck process}

For $c > 0 $, let us consider the Ornstein-Uhlenbeck process: 
\begin{align}
	dX_{t} = dB_{t} - cX_{t}dt \quad (t < T_{0}). \label{}
\end{align}
Its local generator on $(0,\infty)$ is
\begin{align}
	\cL = \frac{1}{2}\frac{d^{2}}{dx^{2}} - cx \frac{d}{dx}, \label{}
\end{align}
and its speed measure and scale function are given by
\begin{align}
	dm(x) = 2 \mathrm{e}^{-cx^{2}} dx, \quad ds(x) = \mathrm{e}^{cx^{2}}dx, \label{} 
\end{align}
respectively.
The boundary $0$ is regular and the boundary $\infty$ is natural.
The bottom of the spectrum $\lambda_{0}$ on $L^{2}(m)$ is $c$.
The eigenfunction $\psi_{q} \ (q \in \bC)$ in \eqref{eq12} is
\begin{align}
	\psi_{q}(x) = x M \left( \frac{q}{2c} + \frac{1}{2},\frac{3}{2}, cx^{2} \right), \label{eq30}
\end{align}
where $M$ is a confluent hypergeometric function (see e.g., \citet[Chapter 6]{Specialfunction}).
We note that
\begin{align}
	\psi_{q}(x) \sim \frac{\Gamma(3/2)c^{q/(2c)}}{\Gamma((q+c)/(2c))} x^{q/c - 1}\mathrm{e}^{cx^{2}} \quad (x \to \infty) \label{eq36}
\end{align}
unless $q = -c(2n+1)$ for a non-negative integer $n$ (see e.g., \citet[p.289]{Specialfunction}).
The process $X$ is positive recurrent under $\bP_{x}^{[\psi_{-c}]}$.
Indeed, since we easily see $\psi_{-c}(x) = x$ from \eqref{eq30}, it follows
\begin{align}
	m^{[\psi_{-c}]}(0,\infty) = \int_{0}^{\infty} \psi_{-c}(x)^{2}dm(x) = 2\int_{0}^{\infty}x^{2}\mathrm{e}^{-cx^{2}}dx < \infty \label{}
\end{align}
and
\begin{align}	
	s^{[\psi_{-c}]}(1,\infty) = \int_{1}^{\infty}\frac{ds(x)}{\psi_{-c}(x)^{2}} = \int_{1}^{\infty}\frac{\mathrm{e}^{cx^{2}}}{x^{2}}dx = \infty. \label{}
\end{align}
For $\lambda \in (0,c)$, an elementary calculus with \eqref{eq36} shows 
\begin{align}
	\int_{1}^{\infty}s^{[\psi_{-\lambda}]}(x,\infty)m^{[\psi_{-\lambda}]}(0,x)^2 ds^{[\psi_{-\lambda}]}(x) < \infty, \label{}
\end{align}
and it follows from Theorem \ref{thm:tailTrivialNeceSuff} that the tail $\sigma$-field $\cT$ is not trivial under $\bP_{x}^{[\psi_{-\lambda}]} \ (\lambda \in (0,c))$.
Perhaps surprisingly, the distributions $\bP_{x}^{[\psi_{-\lambda}]}$ and $\bP_{x}^{[\psi_{-\mu}]} \ (\lambda, \mu \in (0,c), \lambda \neq \mu)$ are mutually absolutely continuous on $\cF_{\infty}$.

\begin{Thm} \label{thm:absoluteContinuityOfOU}
	For $\lambda,\mu \in (0,c)$ with $\lambda \neq \mu$ and $x \geq 0$, the distributions $\bP_{x}^{[\psi_{-\lambda}]}$ and $\bP_{x}^{[\psi_{-\mu}]}$ are mutually absolutely continuous on $\cF_{\infty}$.
	More explicitly, we have
	\begin{align}
		\bP_{0}^{[\psi_{-\mu}]}A = \bE^{[\psi_{-\lambda}]}_{0}[Z_{\lambda,\mu},A] \quad (A \in \cF_{\infty}) \label{}
	\end{align}
	for a random variable $Z_{\lambda,\mu}$ which has the same distribution as
	\begin{align}
		\frac{\Gamma((c-\mu)/(2c))}{\Gamma((c-\lambda)/(2c))} G^{-(\mu-\lambda)/(2c)} \label{eq38}
	\end{align}
	for a Gamma random variable $G$ with parameter $(c-\lambda)/(2c)$, i.e., its probability density is $\mathrm{e}^{-t}t^{(c-\lambda)/(2c)-1} \ (t > 0)$.
\end{Thm}

For the proof, we use the following lemma.

\begin{Lem} \label{lem:asympOfHittingTimeOU}
	For $\lambda \in (0,c)$, we have
	\begin{align}
		T_{y} - \frac{1}{c}\log y \xrightarrow[y \to \infty]{d} -\frac{1}{2c} \log (G/c) \quad \text{under } \bP_{0}^{[\psi_{-\lambda}]}, \label{}
	\end{align}
	where $G$ is a Gamma random variable $G$ with parameter $(c-\lambda)/(2c)$.
\end{Lem}

\begin{proof}
	From Proposition \ref{prop:HittingTimeUnderEntrance} and \eqref{eq36},
	we have
	\begin{align}
		\bE^{[\psi_{-\lambda}]}_{0}[\mathrm{e}^{-q (T_{y} - \log y/c)}] &= \frac{\psi_{-\lambda}(y)}{\psi_{q-\lambda}(y)}y^{q/c} \label{} \\
		&\xrightarrow[y \to \infty]{} \frac{\Gamma((q+c-\lambda)/(2c))}{\Gamma((c-\lambda)/(2c))} c^{-q/(2c)}. \label{eq37}
	\end{align}
	By an easy calculation, we see that \eqref{eq37} is the Laplace transform of $-\log (G/c)/(2c)$.
	Although the random variable $T_{y} - \log y /c$ is not necessarily non-negative, the continuity theorem of the Laplace transform is still valid (see e.g., \citet[Proposition A.1]{FluctuationScalingLimit}),
	and we obtain the desired results.
\end{proof}

We prove Theorem \ref{thm:absoluteContinuityOfOU}.

\begin{proof}[Proof of Theorem \ref{thm:absoluteContinuityOfOU}]
	It is enough to show the case $x = 0$.
	As in the proof of Theorem \ref{thm:TailConstantBM},
	we set $W_{t} := \mathrm{e}^{(\mu-\lambda)t} \psi_{-\mu}(X_{t}) / \psi_{-\lambda}(X_{t})$ and denote the limit random variable $\lim_{t \to \infty}W_{t}$ under $\bP_{0}^{[\psi_{-\lambda}]}$ by $Z$.
	If we can show $Z$ has the same distribution as the random variable in \eqref{eq38},
	we can readily check $\bP_{0}^{[\psi_{-\lambda}]}[Z > 0] = 1$ and $\bE_{0}^{[\psi_{-\lambda}]}Z = 1$.
	This implies the martingale $W_{t}$ is uniformly integrable, and therefore the desired absolute continuity follows.
	We identify the limit distribution of $W_{T_{y}}$.
	By Lemma \ref{lem:asympOfHittingTimeOU} and \eqref{eq36}, we have under $\bP_{0}^{[\psi_{-\lambda}]}$
	\begin{align}
		W_{T_{y}} &= \mathrm{e}^{(\mu-\lambda)T_{y}}\frac{\psi_{-\mu}(y)}{\psi_{-\lambda}(y)} \label{} \\
		&= \mathrm{e}^{(\mu-\lambda)(T_{y} - \log y/c)}\frac{y^{\mu/c}\psi_{-\mu}(y)}{y^{\lambda/c}\psi_{-\lambda}(y)} \label{} \\
		&\xrightarrow[y \to \infty]{d} \frac{\Gamma((c-\lambda)/(2c))}{\Gamma((c-\mu)/(2c))} G^{-(\mu-\lambda)/(2c)}. \label{}
	\end{align}
	The proof is complete.
\end{proof}

\end{document}